\documentclass[pdflatex,sn-mathphys-num]{sn-jnl}

\usepackage{graphicx}%
\usepackage{amsmath,amssymb,amsfonts}%
\usepackage{amsthm}%
\usepackage{mathrsfs}%
\usepackage{textcomp}%

\theoremstyle{thmstyleone}%
\newtheorem{theorem}{Theorem}[section]
\newtheorem{proposition}[theorem]{Proposition}%
\newtheorem{lemma}[theorem]{Lemma}%

\theoremstyle{thmstyletwo}%
\newtheorem{example}{Example}%
\newtheorem{remark}{Remark}%

\theoremstyle{thmstylethree}%
\newtheorem{definition}{Definition}%

\usepackage{mathtools}
\usepackage{upgreek}
\DeclarePairedDelimiter{\abs}{\lvert}{\rvert}
\newcommand{\domD}{\mathscr{D}}
\newcommand{\dd}{\mathrm{d}}

\renewcommand{\Im}{\operatorname{Im}}

\renewcommand{\pi}{\uppi}

\DeclareMathOperator{\arsinh}{arsinh}
\DeclareMathOperator{\OO}{O}
\DeclareMathOperator{\ee}{e}
\DeclareMathOperator{\ii}{i}
\numberwithin{equation}{section}
\begin{document}

\title{Explicit error bounds of the SE and DE
formulas for integrals with logarithmic and algebraic singularity}


\author*[1]{\fnm{Tomoaki} \sur{Okayama}}\email{okayama@hiroshima-cu.ac.jp}
\author[2]{\fnm{Kosei} \sur{Arakawa}}
\author[3]{\fnm{Ryo} \sur{Kamigaki}}
\author[4]{\fnm{Eita} \sur{Yabumoto}}

\affil*[1]{\orgdiv{Graduate School of Information Sciences},
\orgname{Hiroshima City University}, \orgaddress{\street{3-4-1,
Ozuka-higashi}, \city{Asaminami-ku}, \postcode{731-3194},
\state{Hiroshima}, \country{Japan}}}

\affil[2]{\orgdiv{} \orgname{NEC Communication Systems, Ltd.},
\orgaddress{\street{
Mita Kokusai Building, 1-4-28, Mita},
\city{Minato-ku}, \postcode{108-0073}, \state{Tokyo}, \country{Japan}}}

\affil[3]{\orgdiv{} \orgname{Mazda Motor Corporation},
\orgaddress{\street{3-1 Shinchi, Fuchu-cho},
\city{Aki-gun}, \postcode{730-8670}, \state{Hiroshima}, \country{Japan}}}

\affil[4]{\orgdiv{} \orgname{Sobal Corporation},
\orgaddress{\street{Osaki MT Building, Kita-Shinagawa 5-9-11},
\city{Shinagawa-ku}, \postcode{141-0001}, \state{Tokyo}, \country{Japan}}}


\abstract{
The single exponential (SE) and double exponential (DE) formulas are widely
recognized as efficient quadrature formulas for evaluating integrals with
endpoint singularity. For integrals exhibiting algebraic singularity,
explicit error bounds in a computable form have
been provided, enabling computations with guaranteed accuracy.
Such explicit error bounds have also been provided for integrals
exhibiting logarithmic singularity. However, these error bounds have two points
to be discussed. The first point is on overestimation of
divergence speed of logarithmic singularity. The second point is on
the case where there exist both logarithmic and algebraic singularity.
To address these issues, this study provides new error bounds
for integrals with logarithmic and algebraic singularity. Although existing and
new error bounds described above pertain to integrals over the finite
interval, the SE and DE formulas are also applicable to
integrals over the semi-infinite interval. On the basis of
the new results, this study provides new error bounds for
integrals over the semi-infinite interval with logarithmic and algebraic
singularity at the origin.
}




\maketitle

\section{Introduction and summary}\label{sec1}

We are concerned with numerical integration of the integral
\[
 \int_0^T f(t)\dd{t},
\]
where $T>0$ and $f$ may have integrable singularity at the endpoints
of the interval.
The single-exponential (SE) and double-exponential (DE) formulas are
widely recognized as efficient quadrature formulas for
evaluating such an integral~\cite{Mori85}.
These formulas are derived by combining the SE or DE transformation
with the trapezoidal formula.
In the case of the SE formula,
the procedure begins by applying the SE transformation
\begin{equation*}
 t = \psi_1(x)
   = \frac{T}{2}\tanh\left(\frac{x}{2}\right) + \frac{T}{2}
\end{equation*}
to the given integral as
\[
 \int_0^T f(t)\dd{t} = \int_{-\infty}^{\infty} f(\psi_1(x))\psi_1'(x)\dd{x}.
\]
Next, we apply the trapezoidal formula as
\[
 \int_{-\infty}^{\infty} f(\psi_1(x))\psi_1'(x)\dd{x}
\approx h \sum_{k=-\infty}^{\infty} f(\psi_1(kh))\psi_1'(kh),
\]
where $h$ denotes mesh size.
Then, we truncate the infinite sum on the right-hand side at some
integers $M$ and $N$.
The final form of the SE formula is expressed as
\[
 \int_0^T f(t)\dd{t} \approx h\sum_{k=-M}^N f(\psi_1(kh))\psi_1'(kh),
\]
which is also referred to as the tanh rule.
After the formula was proposed~\cite{schwartz69:num_int_anal},
Takahasi and Mori~\cite{bib:DE_TakaMori1974}
derived another formula by replacing the SE transformation with
\begin{equation*}
 t = \phi_1(x)
   = \frac{T}{2}\tanh\left(\frac{\pi}{2}\sinh x\right) + \frac{T}{2},
\end{equation*}
which is referred to as the DE transformation.
The final form of the DE formula is expressed as
\[
 \int_0^T f(t)\dd{t} \approx h\sum_{k=-M}^N f(\phi_1(kh))\phi_1'(kh).
\]
Computable error bounds of the two formulas have been
provided~\cite{OkaMatsuSugi} in the case where
the integrand $f$ has algebraic singularity at the endpoints as
\begin{equation}
 \int_0^T \frac{g(t)}{t^{1-\alpha}(T - t)^{1 - \beta}}\dd{t},
\label{eq:algebraic-singular}
\end{equation}
where $g$ is bounded, and $\alpha$ and $\beta$ are positive constants.
The result has been utilized for the computation library
of verified numerical integration~\cite{bib:Yamanaka10}.

In the case where the integrand $f$ has logarithmic singularity as
\begin{equation}
 \int_0^T g(t)\log t \,\dd{t},
\label{eq:logarithmic-singular}
\end{equation}
computable error bounds of the two formulas have also been
provided~\cite{bib:Okayama14}.
However, two points remain to be discussed regarding these error bounds.
The first point concerns overestimation of divergence speed of
logarithmic singularity. To derive the error bounds,
the following inequality
\begin{equation}
 |\log z| \leq \frac{C}{|z|^{1/(2\pi)}}
\label{eq:log-estimate-by-power}
\end{equation}
was proved with some constant $C$,
and the existing result for~\eqref{eq:algebraic-singular}
was utilized with $\alpha=1 - 1/(2\pi)$ and $\beta = 1$.
Even though the inequality~\eqref{eq:log-estimate-by-power} is
mathematically correct,
the right-hand side is unnecessarily large near the origin,
because the divergence speed of algebraic singularity
is essentially higher than that of logarithmic singularity.
Such an overestimation may result in non-sharp error bounds.
The second point concerns the case where the integrand $f$ has
both logarithmic and algebraic singularity, such as
\begin{equation}
 \int_0^T \frac{g(t)}{t^{1-\alpha}(T - t)^{1 - \beta}}\log t\,\dd{t}.
\label{eq:algebraic-log-singular}
\end{equation}
The provided error bounds for~\eqref{eq:logarithmic-singular}
cannot handle the case as~\eqref{eq:algebraic-log-singular},
because it is assumed that the function $g(t)$ is bounded.

The first contribution of this study is to provide
computable error bounds of the SE and DE formulas
in the case~\eqref{eq:algebraic-log-singular},
to address the aforementioned two points.
Rather than utilizing the error bounds for~\eqref{eq:algebraic-singular},
we perform the error analysis
for the case~\eqref{eq:algebraic-log-singular} directly.
In this direct analysis, in contrast to~\eqref{eq:log-estimate-by-power},
we estimate $|\log z|$ without changing the divergence rate
(see Lemmas~\ref{lem:bound-SE1-log}
and~\ref{lem:bound-DE1-log}
for the estimate of $|\log\psi_1(x+\ii y)|$
and $|\log\phi_1(x+\ii y)|$).
This is key to improving the error bound.
However, this direct analysis requires considerable argument
beyond the estimate of $|\log z|$,
which makes this project challenging.
Completion of all necessary proofs for the analysis
is an important contribution of this paper.
Furthermore, especially for the DE formula,
the selection formulas of $h$, $M$ and $N$
have room for improvement,
because they are not optimally selected
in the previous work~\cite{bib:Okayama14}.
In this study,
we employ the optimal selection formulas of $h$, $M$ and $N$
for the DE formula~\cite{OkaKawai}.
Consequently, we obtain not only sharper error bounds
but also improved convergence profiles, as demonstrated in
numerical experiments.

The second contribution of this study is to provide
computable error bounds in the case of the semi-infinite interval as
\[
 \int_0^{\infty} f(t) \log t\,\dd{t},
\]
where $f$ may have algebraic singularity at the origin.
Two types of the SE and DE formulas were proposed
depending on the decay rate of $f$.
When $f$ decays algebraically as $t\to\infty$,
the SE and DE formulas~\cite{schwartz69:num_int_anal,bib:DE_TakaMori1974}
are expressed as
\begin{align}
 \int_0^{\infty} f(t) \log t\,\dd{t}
&\approx h\sum_{k=-M}^N f(\psi_2(kh))\log(\psi_2(kh))\psi_2'(kh),
\label{eq:SE-formula-2} \\
 \int_0^{\infty} f(t) \log t\,\dd{t}
&\approx h\sum_{k=-M}^N f(\phi_2(kh))\log(\phi_2(kh))\phi_2'(kh),
\label{eq:DE-formula-2}
\end{align}
where $\psi_2$ and $\phi_2$ denote the SE and DE transformations
defined by
\begin{align*}
 \psi_2(x) &= \ee^x,\\
 \phi_2(x) &= \ee^{(\pi/2)\sinh x},
\end{align*}
respectively.
When $f$ decays exponentially as $t\to\infty$,
the SE formula~\cite{okayama17:_error_muham} is derived
by replacing $\psi_2$ in~\eqref{eq:SE-formula-2} with
\begin{align*}
 \psi_3(x) &= \log(1 + \ee^x),\\
\intertext{and the DE formula~\cite{okayama13:_error_sinc_sinc} is derived by
replacing $\phi_2$ in~\eqref{eq:DE-formula-2} with}
 \phi_3(x) &= \log(1 + \ee^{\pi\sinh x}).
\end{align*}
For both types, we provide computable error bounds of the SE and DE
formulas in the same manner as in the case of the finite interval.

The remainder of this paper is organized as follows.
In Section~\ref{sec2},
after reviewing existing theorems,
we present new theorems, which are the main results of this study.
In Section~\ref{sec3},
we present some numerical examples.
In Section~\ref{sec4},
we provide proofs of the main results.

\section{Existing and new results}\label{sec2}

The trapezoidal formula, which is employed in both
the SE and DE formulas, works accurately if the integrand
is analytic on the strip region
\begin{equation*}
\domD_d=\{\zeta\in\mathbb{C} : \abs{\Im\zeta}< d\},
\label{eq:domDd}
\end{equation*}
where $d>0$. Therefore, in the following theorems,
the given integrand $f$ is assumed to be analytic
on the translated domain by the SE or DE transformation,
denoted by
\begin{align*}
 \psi_i(\domD_d) &= \left\{z=\psi_i(\zeta): \zeta\in\domD_d\right\}
\quad (i=1,\,2,\,3),\\
 \phi_i(\domD_d) &= \left\{z=\phi_i(\zeta): \zeta\in\domD_d\right\}
\quad (i=1,\,2,\,3),
\end{align*}
respectively.
Using the notions, we summarize existing and new error bounds below.

\subsection{Existing and new error bounds in the case of the finite interval}

The existing error bounds
for the SE and DE formulas are expressed as follows.

\begin{theorem}[Okayama~{\cite[Theorem 2]{bib:Okayama14}}]
\label{thm:exist-SE}
Let $K$ and $d$ be positive constants with $d<\pi$.
Assume that $f$ is analytic on $\psi_1(\domD_d)$,
and satisfies
\begin{equation}
 |f(z)|\leq K|\log z|
\label{eq:f-bound-log}
\end{equation}
for all $z\in\psi_1(\domD_d)$.
Let $\gamma=(2\pi - 1)/(2\pi)$, let $n$ be a positive integer,
let $h$ be selected by
\[
 h = \sqrt{\frac{2\pi d}{\gamma n}},
\]
and let $M$ and $N$ be selected by
\[
 M = n,\quad N = \left\lceil \gamma n\right\rceil.
\]
Then, it holds that
\[
\left\lvert
\int_0^T f(t)\dd{t}
- h \sum_{k=-M}^N f(\psi_1(kh))\psi_1'(kh)
\right\rvert
\leq C \ee^{-\sqrt{2\pi d \gamma n}},
\]
where $C$ is a constant independent of $n$, expressed as
\[
 C
=\frac{2KT}{\gamma \cos^{1/(2\pi)}(d/2)}
\sqrt{\pi^2 + \left\{\log\left(\frac{T}{\cos(d/2)}\right)\right\}^2}
\left[
\frac{2}{(1 - \ee^{-\sqrt{2\pi d \gamma}})\cos^{\gamma+1}(d/2)}
+ 1
\right].
\]
\end{theorem}

\begin{theorem}[Okayama~{\cite[Theorem 3]{bib:Okayama14}}]
\label{thm:exist-DE}
Let $K$ and $d$ be positive constants with $d<\pi/2$.
Assume that $f$ is analytic on $\phi_1(\domD_d)$,
and satisfies~\eqref{eq:f-bound-log} for all $z\in\phi_1(\domD_d)$.
Let $\gamma=(2\pi - 1)/(2\pi)$, let $n$ be a positive integer,
let $h$ be selected by
\[
 h = \frac{\log(4 d n/\gamma)}{n},
\]
and let $M$ and $N$ be selected by
\[
 M = n, \quad N = n - \left\lfloor \frac{\log(1/\gamma)}{h}\right\rfloor.
\]
Then, it holds that
\[
\left\lvert
\int_0^T f(t)\dd{t}
- h \sum_{k=-M}^N f(\phi_1(kh))\phi_1'(kh)
\right\rvert
\leq C \ee^{-2\pi d n/\log(4 d n/\gamma)},
\]
where $C$ is a constant independent of $n$, expressed as
\[
 C
=\frac{2KT c_d^{1/(2\pi)}}{\gamma}
\sqrt{\pi^2 + \left\{\log\left(T c_d\right)\right\}^2}
\left[
\frac{2c_d^{\gamma+1}}{(1 - \ee^{-\pi\gamma\ee/2})\cos d}
+ \ee^{\pi/2}
\right],
\]
where $c_d = 1/\cos((\pi/2)\sin d)$.
\end{theorem}

In this paper,
we present the following error bounds for the SE and DE formulas.
Their proofs are provided in Section~\ref{sec:proof-finite}.

\begin{theorem}
\label{thm:new-SE1}
Let $K$, $\alpha$, $\beta$ and $d$ be positive constants with $d<\pi$.
Assume that $f$ is analytic on $\psi_1(\domD_d)$,
and satisfies
\begin{equation}
 |f(z)|\leq K|z|^{\alpha-1}|T - z|^{\beta - 1}|\log z|
\label{eq:f-bound-log-algebraic}
\end{equation}
for all $z\in\psi_1(\domD_d)$.
Let $\mu=\min\{\alpha,\beta\}$,
let $n$ be a positive integer,
let $h$ be selected by
\begin{equation}
 h = \sqrt{\frac{2\pi d}{\mu n}},
\label{eq:standard-h-SE}
\end{equation}
and let $M$ and $N$ be selected by
\begin{equation}
M=\left\lceil\frac{\mu}{\alpha}n\right\rceil,\quad
N=\left\lceil\frac{\mu}{\beta}n\right\rceil.
\label{eq:def-MN-SE}
\end{equation}
Let $n$ be taken sufficiently large
so that $n\geq 1/(2\pi d \mu)$ is satisfied.
Then, it holds that
\[
\left\lvert
\int_0^T f(t)\dd{t}
- h \sum_{k=-M}^N f(\psi_1(kh))\psi_1'(kh)
\right\rvert
\leq C \sqrt{n} \ee^{-\sqrt{2\pi d \mu n}},
\]
where $C$ is a constant independent of $n$, expressed as
\[
 C
=\frac{KT^{\alpha+\beta-1}}{\mu}
\left[
\frac{4|\log T|\cos(d/2) + 2l_{\mu}}{(1 - \ee^{-\sqrt{2\pi d \mu}})\cos^{\alpha+\beta+1}(d/2)}
+ 2|\log T| + l_{\mu} + \sqrt{\frac{2\pi d}{\mu}}
\right],
\]
where $l_{\mu}=2\log 2 + (1/\mu)$.
\end{theorem}

\begin{theorem}
\label{thm:new-DE1}
Let $K$, $\alpha$, $\beta$ and $d$ be positive constants with $d<\pi/2$.
Assume that $f$ is analytic on $\phi_1(\domD_d)$,
and satisfies~\eqref{eq:f-bound-log-algebraic} for all $z\in\phi_1(\domD_d)$.
Let $\mu=\min\{\alpha,\beta\}$, let $n$ be a positive integer,
let $h$ be selected by
\begin{equation}
 h = \frac{\arsinh(2dn/\mu)}{n}, \label{eq:improve-h-DE}
\end{equation}
and let $M$ and $N$ be selected by
\begin{equation}
 M=\left\lceil\frac{1}{h}\arsinh\left(\frac{\mu}{\alpha}q\left(\frac{2dn}{\mu}\right)\right)\right\rceil,\quad
 N= \left\lceil\frac{1}{h}\arsinh\left(\frac{\mu}{\beta}q\left(\frac{2dn}{\mu}\right)\right)\right\rceil,
\label{eq:def-MN-DE}
\end{equation}
where $q(x)=x/\arsinh x$.
Let $n$ be taken sufficiently large
so that $n\geq \mu\sinh(1)/(2d)$ and $h\leq \pi d$ are satisfied.
Then, it holds that
\[
\left\lvert
\int_0^T f(t)\dd{t}
- h \sum_{k=-M}^N f(\phi_1(kh))\phi_1'(kh)
\right\rvert
\leq C n\ee^{-2\pi d n/\arsinh(2 d n/\mu)},
\]
where $C$ is a constant independent of $n$, expressed as
\[
 C
=\frac{KT^{\alpha+\beta-1}}{\mu}
\left[
\frac{c_d^{\alpha+\beta}(4|\log T| \cos d + 2l_\mu c_d)}{(1 - \ee^{-\pi\mu q(2d/\mu)})\cos^2 d}
+ 2|\log T| + l_{\mu} + \frac{2\pi d}{\mu}
\right],
\]
where $c_d = 1/\cos((\pi/2)\sin d)$ and $l_{\mu}=2\log 2 + (1/\mu)$.
\end{theorem}

Let us compare the existing and new theorems.
For a fair comparison,
we set $\alpha=\beta=1$ in the new theorems here,
so that~\eqref{eq:f-bound-log-algebraic} becomes
the same condition as~\eqref{eq:f-bound-log}.
According to Theorem~\ref{thm:exist-SE} (existing theorem),
the convergence rate of the SE formula is
$\OO(\ee^{-\sqrt{2\pi d \gamma n}})$,
where $\gamma = (2\pi  - 1)/(2\pi)$.
In contrast,
Theorem~\ref{thm:new-SE1} (new theorem)
states that the convergence rate of the SE formula is
$\OO(\sqrt{n}\ee^{-\sqrt{2\pi d  n}})$, which is higher than
$\OO(\ee^{-\sqrt{2\pi d \gamma n}})$
(note that $\gamma<1$).

However, we must note that the aforementioned convergence rates are
given with respect to $n$, not with respect to
the total number of function evaluations ($M+N+1$).
For this reason, we cannot conclude immediately
from the theorems above that
the SE formula in the new theorem converges more rapidly
than that in the existing theorem.
In fact, the total number of function evaluations
($M+N+1$) differs according to the theorems;
$M+N+1=2n+1$ in Theorem~\ref{thm:new-SE1},
whereas $M+N+1=n + \lceil\gamma n\rceil + 1$
in Theorem~\ref{thm:exist-SE}, which is less than $2n+1$.

Similarly, in the case of the DE formula as well,
it is difficult to judge which is better from Theorems~\ref{thm:exist-DE}
and~\ref{thm:new-DE1}.
To compare the convergence profiles with respect to $M+N+1$ practically,
numerical experiments are useful, which are demonstrated in the next section.


\subsection{New error bounds in the case of the semi-finite interval}

Here, we present new error bounds
for the SE and DE formulas
in the case where the integration interval is $(0,\infty)$,
i.e., the semi-infinite interval.
Their proofs are provided in Section~\ref{sec:proof-semi-infinite}.

We consider two cases depending on the decay rate of
the integrand: algebraic decay or exponential decay.
First, we consider the case of algebraic decay.

\begin{theorem}
\label{thm:new-SE2}
Let $K$, $\alpha$, $\beta$ and $d$ be positive constants with $d<\pi/2$.
Assume that $f$ is analytic on $\psi_2(\domD_d)$,
and satisfies
\begin{equation}
 |f(z)|\leq K\frac{|z|^{\alpha-1}}{|1+z^2|^{(\alpha+\beta)/2}}|\log z|
\label{eq:f-bound-log-algebraic-2}
\end{equation}
for all $z\in\psi_2(\domD_d)$.
Let $\mu=\min\{\alpha,\beta\}$,
let $n$ be a positive integer,
let $h$ be selected by~\eqref{eq:standard-h-SE},
and let $M$ and $N$ be selected by~\eqref{eq:def-MN-SE}.
Let $n$ be taken sufficiently large
so that $n\geq 1/(2\pi d \mu)$ is satisfied.
Then, it holds that
\[
\left\lvert
\int_0^{\infty} f(t)\dd{t}
- h \sum_{k=-M}^N f(\psi_2(kh))\psi_2'(kh)
\right\rvert
\leq C \sqrt{n} \ee^{-\sqrt{2\pi d \mu n}},
\]
where $C$ is a constant independent of $n$, expressed as
\[
 C
=\frac{2K}{\mu^2}
\left[
\frac{2(1 + \mu d)}{(1 - \ee^{-\sqrt{2\pi d \mu}})\cos^{(\alpha+\beta)/2}d}
+ \sqrt{2\pi d \mu} + 1
\right].
\]
\end{theorem}

\begin{theorem}
\label{thm:new-DE2}
Let $K$, $\alpha$, $\beta$ and $d$ be positive constants with $d<\pi/2$.
Assume that $f$ is analytic on $\phi_2(\domD_d)$,
and satisfies~\eqref{eq:f-bound-log-algebraic-2} for all $z\in\phi_2(\domD_d)$.
Let $\mu=\min\{\alpha,\beta\}$, let $n$ be a positive integer,
let $h$ be selected by
\begin{equation}
 h = \frac{\arsinh(4dn/\mu)}{n}, \label{eq:improve-h-DE2}
\end{equation}
and let $M$ and $N$ be selected by
\begin{equation}
 M=\left\lceil\frac{1}{h}\arsinh\left(\frac{\mu}{\alpha}q\left(\frac{4dn}{\mu}\right)\right)\right\rceil,\quad
 N= \left\lceil\frac{1}{h}\arsinh\left(\frac{\mu}{\beta}q\left(\frac{4dn}{\mu}\right)\right)\right\rceil,
\label{eq:def-MN-DE2}
\end{equation}
where $q(x)=x/\arsinh x$.
Let $n$ be taken sufficiently large
so that $n\geq \mu\sinh(1)/(4d)$ and $h\leq \pi d$ are satisfied.
Then, it holds that
\[
\left\lvert
\int_0^{\infty} f(t)\dd{t}
- h \sum_{k=-M}^N f(\phi_2(kh))\phi_2'(kh)
\right\rvert
\leq C n\ee^{-2\pi d n/\arsinh(4 d n/\mu)},
\]
where $C$ is a constant independent of $n$, expressed as
\[
 C
=\frac{2K}{\mu^2}
\left[
\frac{(2 + \pi \mu \cos d)c_d^{(\alpha+\beta)/2}}{(1 - \ee^{-\pi\mu q(4d/\mu)/2})\cos^2 d}
+ 2\pi d + 1
\right],
\]
where $c_d = 1/\cos((\pi/2)\sin d)$.
\end{theorem}

Next, we consider the case of exponential decay.

\begin{theorem}
\label{thm:new-SE3}
Let $K$, $\alpha$, $\beta$ and $d$ be positive constants with
$\alpha\leq 1$ and $d<\pi$.
Assume that $f$ is analytic on $\psi_3(\domD_d)$,
and satisfies
\begin{equation}
 |f(z)|\leq K
\left|\frac{z}{1+z}\right|^{\alpha-1}|\ee^{-z}|^{\beta}|\log z|
\label{eq:f-bound-log-algebraic-3}
\end{equation}
for all $z\in\psi_3(\domD_d)$.
Let $\mu=\min\{\alpha,\beta\}$,
let $n$ be a positive integer,
let $h$ be selected by~\eqref{eq:standard-h-SE},
and let $M$ and $N$ be selected by~\eqref{eq:def-MN-SE}.
Let $n$ be taken sufficiently large
so that $n\geq 1/(2\pi d \mu)$ is satisfied.
Then, it holds that
\[
\left\lvert
\int_0^{\infty} f(t)\dd{t}
- h \sum_{k=-M}^N f(\psi_3(kh))\psi_3'(kh)
\right\rvert
\leq C \sqrt{n} \ee^{-\sqrt{2\pi d \mu n}},
\]
where $C$ is a constant independent of $n$, expressed as
\begin{align*}
 C
&=\frac{2K}{\mu^2}
\Biggl[
\frac{2\tilde{L}_d^{1-\alpha}\tilde{c}^{\alpha+\beta}_d\left\{(1+\tilde{c}_d)(1+\mu d)-\mu\log(\log 2)\log(2+\tilde{c}_d)\right\}}{(1 - \ee^{-\sqrt{2\pi d \mu}})\log(2+\tilde{c}_d)}
\Biggr.\\
&\Biggl. \quad\quad\quad\quad + \ee^{\pi(1-\alpha)/12}\left\{
\sqrt{2\pi d \mu} + 1 - \mu \log(\log 2)
\right\}
\Biggr],
\end{align*}
where $\tilde{c}_d=1/\cos(d/2)$ and
\begin{equation}
 \tilde{L}_d
 = \frac{1 + \log(2 + \tilde{c}_d)}{\log(2 + \tilde{c}_d)}(1 + \tilde{c}_d).
\label{eq:tilde-L-d}
\end{equation}
\end{theorem}

\begin{theorem}
\label{thm:new-DE3}
Let $K$, $\alpha$, $\beta$ and $d$ be positive constants
with $\alpha\leq 1$ and $d<\pi/2$.
Assume that $f$ is analytic on $\phi_3(\domD_d)$,
and satisfies~\eqref{eq:f-bound-log-algebraic-2} for all $z\in\phi_3(\domD_d)$.
Let $\mu=\min\{\alpha,\beta\}$, let $n$ be a positive integer,
let $h$ be selected by~\eqref{eq:improve-h-DE},
and let $M$ and $N$ be selected by~\eqref{eq:def-MN-DE},
where $q(x)=x/\arsinh x$.
Let $n$ be taken sufficiently large
so that $n\geq \mu\sinh(1)/(2d)$ and $h\leq \pi d$ are satisfied.
Then, it holds that
\[
\left\lvert
\int_0^{\infty} f(t)\dd{t}
- h \sum_{k=-M}^N f(\phi_3(kh))\phi_3'(kh)
\right\rvert
\leq C n\ee^{-2\pi d n/\arsinh(2 d n/\mu)},
\]
where $C$ is a constant independent of $n$, expressed as
\begin{align*}
 C
&=\frac{2K}{\mu^2}
\Biggl[
\frac{2 L_d^{1-\alpha}c_d^{\alpha+\beta}\left\{
(1+c_d)(1+d)(1+\pi\mu\cos d) - \mu\log(\log 2)\log(2+c_d)\cos d
\right\}}{(1 - \ee^{-\pi\mu q(2d/\mu)})\log(2+c_d)\cos^2 d}
\Biggr.\\
& \quad\quad\quad\quad
+ \ee^{\pi(1-\alpha)/12}\left\{2\pi d + 1 - \mu\log(\log 2)\right\}
\Biggr],
\end{align*}
where $c_d = 1/\cos((\pi/2)\sin d)$ and
\begin{equation}
 L_d
 = \frac{1 + \log(2 + {c}_d)}{\log(2 + {c}_d)}(1 + {c}_d).
\label{eq:L-d}
\end{equation}
\end{theorem}

\section{Numerical experiments}
\label{sec3}

In this section, we present numerical results
for integrals with logarithmic singularity.
We implemented the programs in C
with double-precision arithmetic.

We first consider the following integral over the finite interval,
where algebraic singularity does not exist.

\begin{example}
\label{ex1}
Consider the following integral
\[
 \int_0^1 \frac{\log t}{1+t}\dd{t} = - \frac{\pi^2}{12}.
\]
\end{example}

In this case,
the assumptions of Theorem~\ref{thm:exist-SE}
are fulfilled with $K=1 + \ee$ and $d=3$,
and those of Theorem~\ref{thm:exist-DE}
are fulfilled with $K=3\sqrt{2}$ and $d=\pi/3$.
Furthermore,
the assumptions of Theorem~\ref{thm:new-SE1}
are fulfilled with $K=1 + \ee$, $\alpha=\beta=1$ and $d=3$,
and those of Theorem~\ref{thm:new-DE1}
are fulfilled with $K=3\sqrt{2}$, $\alpha=\beta=1$ and $d=\pi/3$.
The results are shown in Fig.~\ref{fig:1}.
In the figure,
the error bounds given in Theorems~\ref{thm:exist-SE}--\ref{thm:new-DE1}
enclose corresponding observed errors.
The convergence profile
of the SE formula of Theorem~\ref{thm:new-SE1}
is almost identical to that of Theorem~\ref{thm:exist-SE}.
In contrast,
the convergence profile
of the DE formula of Theorem~\ref{thm:new-DE1}
is superior to that of Theorem~\ref{thm:exist-DE}.
The primary reason for this improvement may be due to employment of
the optimal selection formulas of $h$, $M$ and $N$~\cite{OkaKawai}.

The next example is also an integral over the finite interval,
but algebraic singularity exists at the origin.
For this reason, we cannot use Theorems~\ref{thm:exist-SE}
and~\ref{thm:exist-DE} for this example.

\begin{example}
\label{ex2}
Consider the following integral
\[
 \int_0^1 \frac{\log t}{\sqrt{t}(1+t)}\dd{t} = - 4\mathrm{G},
\]
where $\mathrm{G}$ is Catalan's constant.
\end{example}

In this case,
the assumptions of Theorem~\ref{thm:new-SE1}
are fulfilled with $K=1 + \ee$, $\alpha=1/2$, $\beta=1$ and $d=3$,
and those of Theorem~\ref{thm:new-DE1}
are fulfilled with $K=3\sqrt{2}$, $\alpha=1/2$, $\beta=1$ and $d=\pi/3$.
The results are shown in Fig.~\ref{fig:2}.
In this figure as well,
the error bounds given in Theorems~\ref{thm:new-SE1} and~\ref{thm:new-DE1}
enclose corresponding observed errors.

The next example is an integral over the semi-infinite interval,
where the integrand decays polynomially.

\begin{example}
\label{ex3}
Consider the following integral
\[
\int_0^{\infty} \frac{\log t}{t^{1/3}(1+t^2)}\dd{t}
 = - \frac{\pi^2}{6}.
\]
\end{example}

In this case,
the assumptions of both Theorems~\ref{thm:new-SE2}
and~\ref{thm:new-DE2}
are fulfilled with $K=1$, $\alpha=2/3$, $\beta=4/3$ and $d=3/2$.
The results are shown in Fig.~\ref{fig:3}.
In this figure as well,
the error bounds given in Theorems~\ref{thm:new-SE2} and~\ref{thm:new-DE2}
enclose corresponding observed errors.

The final example is the integral over the semi-infinite interval,
where the integrand decays exponentially.

\begin{example}
\label{ex4}
Consider the following integral
\[
\int_0^{\infty} \frac{\ee^{-t}\log t}{\sqrt{t}}\dd{t}
 = - \sqrt{\pi}(\gamma + 2\log 2),
\]
where $\gamma$ is Euler's constant.
\end{example}

In this case,
the assumptions of Theorem~\ref{thm:new-SE3}
are fulfilled with $K=2\pi/3$, $\alpha=1/2$, $\beta=1$ and $d=3$,
and those of Theorem~\ref{thm:new-DE3}
are fulfilled with $K=2\pi/3$, $\alpha=1/2$, $\beta=1$ and $d=3/2$.
The results are shown in Fig.~\ref{fig:4}.
In this figure as well,
the error bounds given in Theorems~\ref{thm:new-SE2} and~\ref{thm:new-DE2}
enclose corresponding observed errors.

\begin{figure}[htbp]
\centering
\includegraphics[scale=.75]{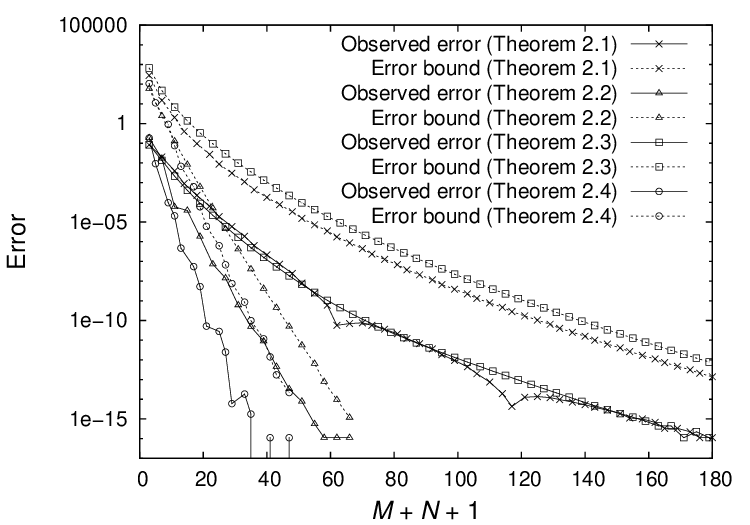}
\caption{Observed error and error bound for the integral of Example~\ref{ex1}.}
\label{fig:1}
\end{figure}
\begin{figure}[htbp]
\centering
\includegraphics[scale=.75]{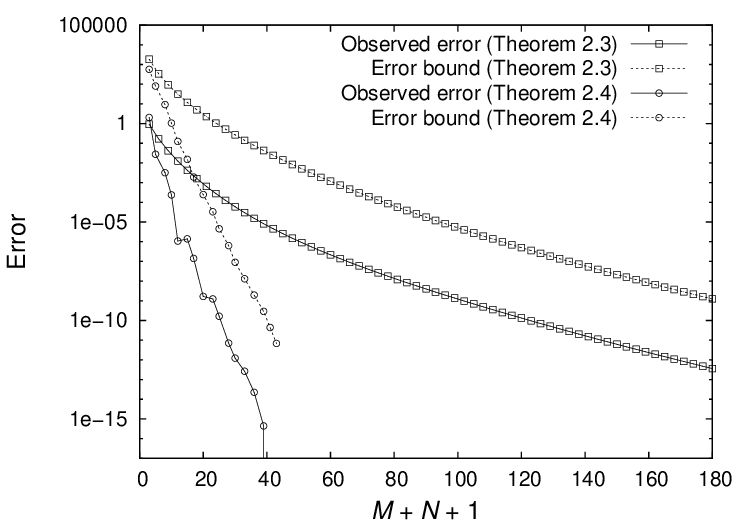}
\caption{Observed error and error bound for the integral of Example~\ref{ex2}.}
\label{fig:2}
\end{figure}
\begin{figure}[htbp]
\centering
\includegraphics[scale=.75]{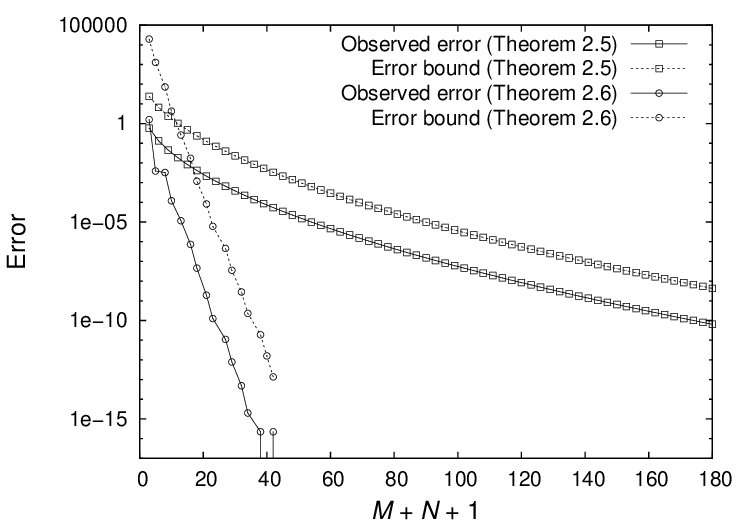}
\caption{Observed error and error bound for the integral of Example~\ref{ex3}.}
\label{fig:3}
\end{figure}
\begin{figure}[htbp]
\centering
\includegraphics[scale=.75]{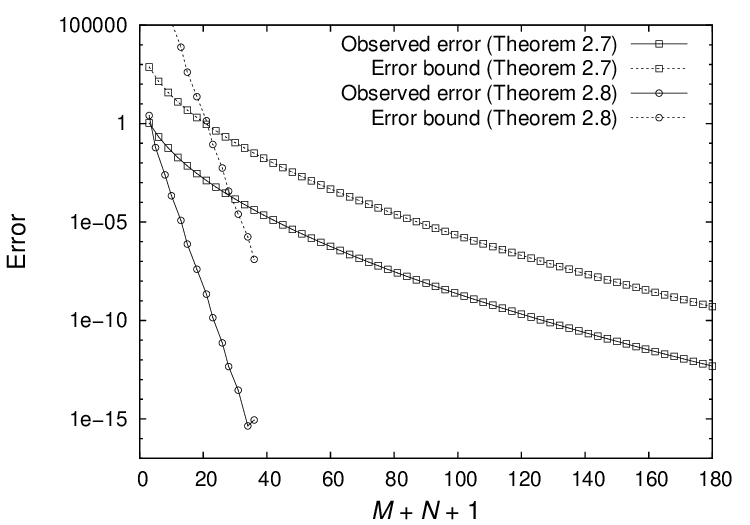}
\caption{Observed error and error bound for the integral of Example~\ref{ex4}.}
\label{fig:4}
\end{figure}

\begin{remark}
The gap between the observed error and theoretical error bound
is mainly due to the terms involving $d$ in the constant $C$,
such as $1/\cos(d/2)$ (which becomes large when $d$ is close to $\pi$)
and $1/\cos((\pi/2)\sin d)$ (which becomes large when $d$ is close to $\pi/2$).
Reducing the value of $d$ can decrease this gap.
However, doing so worsens the convergence rate,
because $d$ directly affects
the exponential term of the convergence rate,
such as $\ee^{-\sqrt{2\pi d \mu n}}$
or $\ee^{-2\pi d n/\arsinh(2dn/\mu)}$.
To obtain a well-balanced error bound,
it is desirable to develop a systematic method for choosing an appropriate value of $d$.
This remains an open topic for future research.
\end{remark}

\section{Proofs}
\label{sec4}

In this section, we provide proofs
for Theorems~\ref{thm:new-SE1}--\ref{thm:new-DE3}.

\subsection{In the case of the finite interval}
\label{sec:proof-finite}

Here, we prove theorems in the case of the finite interval,
i.e., Theorems~\ref{thm:new-SE1} (for the SE formula)
and~\ref{thm:new-DE1} (for the DE formula).

\subsubsection{Proof of Theorem~\ref{thm:new-SE1}}

We begin by analyzing the error of the SE formula.
Applying the SE transformation $t=\psi_1(x)$ and
putting $F(x)=f(\psi_1(x))\psi_1'(x)$,
we rewrite the error of the SE formula as
\[
\left|
 \int_{0}^T f(t)\dd{t} - h\sum_{k=-M}^N f(\psi_1(kh))\psi_1'(kh)
\right|
=
\left|
 \int_{-\infty}^{\infty}F(x)\dd{x} - h\sum_{k=-M}^N F(kh)
\right|.
\]
We now divide the right-hand side into two terms as
\begin{align*}
 \left|
 \int_{-\infty}^{\infty}F(x)\dd{x} - h\sum_{k=-M}^N F(kh)
\right|
&\leq\left|
\int_{-\infty}^{\infty}F(x)\dd{x} - h\sum_{k=-\infty}^{\infty} F(kh)
\right|\\
&\quad +
\left|
h\sum_{k=-\infty}^{-M-1} F(kh)
+h\sum_{k=N+1}^{\infty} F(kh)
\right|.
\end{align*}
The first and second terms are called
the discretization and truncation errors, respectively.
The following function space is important
to analyze the discretization error.

\begin{definition}
Let $d$ be a positive constant.
Then, $\mathbf{B}(\domD_d)$ denotes the family of all functions $F$
that are analytic on $\domD_d$, such that
\begin{equation}
 \lim_{x\to\pm\infty}\int_{-d}^d |F(x+\ii y)|\dd{y}=0,
\label{eq:F-imag-int-0}
\end{equation}
and such that $\mathcal{N}(F,d)<\infty$, where
\[
 \mathcal{N}(F,d)=\lim_{y\to d - 0}
\int_{-\infty}^{\infty}
\left\{|F(x + \ii y)| + |F(x - \ii y)|\right\}\dd{x}.
\]
\end{definition}

If the integrand $F$ belongs to $\mathbf{B}(\domD_d)$,
the discretization error is bounded as follows.

\begin{theorem}[classical; cf.\ Stenger~{\cite[Theorem~3.4]{stenger81:_numer}}]
Assume that $F\in\mathbf{B}(\domD_d)$. Then, it holds that
\begin{equation}
 \left|
\int_{-\infty}^{\infty}F(x)\dd{x} - h\sum_{k=-\infty}^{\infty} F(kh)
\right|
\leq \mathcal{N}(F,d)\frac{\ee^{-2\pi d/h}}{1 - \ee^{-2\pi d/h}}.
\label{eq:discretization-error}
\end{equation}
\end{theorem}

Therefore, our main task is to show $F\in\mathbf{B}(\domD_d)$.
To this end, the following bounds are required.

\begin{lemma}[Okayama et al.~{\cite[Lemma~4.21]{OkaMatsuSugi}}]
\label{lem:bound-SE-1p-exp}
For all real numbers $x$ and $y$ with $|y|<\pi$,
it holds that
\begin{align*}
\left|
\frac{1}{1+\ee^{x + \ii y}}
\right|
&\leq \frac{1}{(1+\ee^{x})\cos(y/2)},\\
\left|
\frac{1}{1+\ee^{-(x + \ii y)}}
\right|
&\leq \frac{1}{(1+\ee^{-x})\cos(y/2)}.
\end{align*}
\end{lemma}

\begin{lemma}
\label{lem:bound-SE1-log}
Let $T$ be a positive real number.
For all real numbers $x$ and $y$ with $|y|<\pi$,
it holds that
\begin{align*}
\left|
\log\left(\frac{T}{1+\ee^{-(x+\ii y)}}\right)
\right|
&\leq |\log T| + \frac{1}{\cos(y/2)}\log(1+\ee^{-x}).
\end{align*}
\end{lemma}
\begin{proof}
First, it holds that
\begin{align*}
\left|
\log\left(\frac{T}{1+\ee^{-(x + \ii y)}}\right)
\right|
&=   \left|\log T - \log(1 + \ee^{-(x + \ii y)})\right|\\
&\leq |\log T| + |\log(1 + \ee^{-(x+\ii y)})|.
\end{align*}
Furthermore, noting $(\log(1+\ee^{-\zeta}))'= - 1/(1+\ee^{\zeta})$
and using Lemma~\ref{lem:bound-SE-1p-exp},
we have
\begin{align*}
 |\log(1+\ee^{-(x+\ii y)})|
&=
\left|\int_x^{\infty}\frac{1}{1+\ee^{t+\ii y}}\dd{t}\right|\\
&\leq \int_x^{\infty}\frac{1}{|1+\ee^{t+\ii y}|}\dd{t}\\
&\leq \frac{1}{\cos(y/2)}\int_{x}^{\infty}\frac{1}{(1+\ee^{t})}\dd{t}\\
&= \frac{1}{\cos(y/2)}\log(1+\ee^{-x}),
\end{align*}
from which we obtain the desired inequality.
\end{proof}

Using these bounds, we show the following lemma.

\begin{lemma}
\label{lem:SE1-discretization}
Let $K$, $\alpha$, $\beta$ and $d$ be positive constants with $d<\pi$.
Assume that $f$ is analytic on $\psi_1(\domD_d)$, and
satisfies~\eqref{eq:f-bound-log-algebraic} for all $z\in\psi_1(\domD_d)$.
Let $\mu=\min\{\alpha,\beta\}$. Then, putting
$F(x)=f(\psi_1(x))\psi_1'(x)$, we have~\eqref{eq:discretization-error},
where
\[
\mathcal{N}(F,d)
\leq\frac{2KT^{\alpha+\beta-1}}{\mu\cos^{\alpha+\beta}(d/2)}
\left\{2|\log T|+
 \left(2\log 2 + \frac{1}{\mu}\right)\frac{1}{\cos(d/2)}\right\}.
\]
\end{lemma}
\begin{proof}
It suffices to show that $F\in\mathbf{B}(\domD_d)$.
Because $f(\psi_1(\cdot))$ is analytic on $\domD_d$ and
$\psi_1'$ is analytic on $\domD_{\pi}$,
$F$ is analytic on $\domD_d$.
Next, we show~\eqref{eq:F-imag-int-0}.
From~\eqref{eq:f-bound-log-algebraic}
and Lemmas~\ref{lem:bound-SE-1p-exp}
and~\ref{lem:bound-SE1-log}, it holds
for $\zeta=x+\ii y\in \domD_d$ that
\begin{align}
 |F(\zeta)|&\leq
\frac{KT^{\alpha+\beta-1}}{|1+\ee^{-\zeta}|^{\alpha}|1+\ee^{\zeta}|^{\beta}}
\left|\log\left(\frac{T}{1+\ee^{-\zeta}}\right)\right|\nonumber\\
&\leq \frac{KT^{\alpha+\beta-1}}{(1+\ee^{-x})^{\alpha}\cos^{\alpha}(y/2)
(1+\ee^{x})^{\beta}\cos^{\beta}(y/2)}
\left(|\log T| + \frac{1}{\cos(y/2)}\log(1+\ee^{-x})\right).
\label{eq:bound-F-SE1}
\end{align}
Using this inequality, for $x\geq 0$ we have
\begin{align*}
&     \int_{-d}^d\left|F(x + \ii y)\right|\dd y\\
&\leq\frac{KT^{\alpha+\beta-1}}{(1+\ee^{-x})^{\alpha}(1+\ee^{x})^{\beta}}
\int_{-d}^d
\frac{1}{\cos^{\alpha+\beta}(y/2)}
\left(|\log T| + \frac{1}{\cos(y/2)}\log(1+\ee^{-x})\right)
 \dd{y}\\
&\leq\frac{KT^{\alpha+\beta-1}}{(1+\ee^{-x})^{\alpha}(1+\ee^{x})^{\beta}}
\int_{-d}^d
\frac{1}{\cos^{\alpha+\beta}(y/2)}
\left(|\log T| + \frac{1}{\cos(y/2)}\log(1+\ee^{-0})\right)
 \dd{y}\\
&\to 0 \cdot
\int_{-d}^d\frac{1}{\cos^{\alpha+\beta}(y/2)}
\left(|\log T| + \frac{1}{\cos(y/2)}\cdot \log 2\right)\dd{y}
\quad (x \to \infty),
\end{align*}
and for $x < 0$, using $\log(1+\ee^{-x})\leq -x+1$, we have
\begin{align*}
&     \int_{-d}^d\left|F(x + \ii y)\right|\dd y\\
&\leq\frac{(-x+1) KT^{\alpha+\beta-1}}{(1+\ee^{-x})^{\alpha}(1+\ee^{x})^{\beta}}
\int_{-d}^d\frac{1}{\cos^{\alpha+\beta}(y/2)}
\left(\frac{|\log T|}{(-x+1)} + \frac{1}{\cos(y/2)}\frac{\log(1+\ee^{-x})}{(-x+1)}\right)\dd{y}\\
&\leq\frac{(-x+1) KT^{\alpha+\beta-1}}{(1+\ee^{-x})^{\alpha}(1+\ee^{x})^{\beta}}
\int_{-d}^d\frac{1}{\cos^{\alpha+\beta}(y/2)}
\left(\frac{|\log T|}{(-0+1)} + \frac{1}{\cos(y/2)}\cdot 1\right)\dd{y}\\
&\to 0 \cdot
\int_{-d}^d\frac{1}{\cos^{\alpha+\beta}(y/2)}
\left(|\log T| + \frac{1}{\cos(y/2)}\cdot 1\right)\dd{y}
\quad (x \to -\infty),
\end{align*}
which shows~\eqref{eq:F-imag-int-0}.
Finally, we estimate $\mathcal{N}(F,d)$.
Using~\eqref{eq:bound-F-SE1},
we have
\begin{align*}
&\int_{-\infty}^{\infty} \left\{|F(x + \ii y)|+|F(x - \ii y)|\right\}\dd x\\
&\leq\int_{-\infty}^{\infty}
\frac{KT^{\alpha+\beta-1}}{(1+\ee^{-x})^{\alpha}
(1+\ee^{x})^{\beta}\cos^{\alpha+\beta}(y/2)}
\left(|\log T| + \frac{\log(1+\ee^{-x})}{\cos(y/2)}\right)\dd{x}\\
&\quad +
\int_{-\infty}^{\infty}
\frac{KT^{\alpha+\beta-1}}{(1+\ee^{-x})^{\alpha}
(1+\ee^{x})^{\beta}\cos^{\alpha+\beta}(-y/2)}
\left(|\log T| + \frac{\log(1+\ee^{-x})}{\cos(-y/2)}\right)\dd{x}\\
&\leq\frac{2KT^{\alpha+\beta-1}}{\cos^{\alpha+\beta}(d/2)}
\int_{-\infty}^{\infty}
\frac{1}{(1+\ee^{-x})^{\alpha}(1+\ee^{x})^{\beta}}
\left(|\log T| + \frac{\log(1+\ee^{-x})}{\cos(d/2)}\right)\dd{x},
\end{align*}
where $\cos(d/2)\leq\cos(\pm y/2)$ is used,
which holds for $y\in [-d,d]$ (note that $d<\pi$).
Therefore, the inequality remains valid
when taking the limit $y\to d - 0$.
For the first term of the integral, using $\mu=\min\{\alpha,\beta\}$,
we have
\begin{align*}
\int_{-\infty}^{\infty}
 \frac{|\log T|}{(1+\ee^{-x})^{\alpha}(1+\ee^{x})^{\beta}}\dd{x}
&\leq \int_{-\infty}^{\infty}
 \frac{|\log T|}{(1+\ee^{-x})^{\mu}(1+\ee^{x})^{\mu}}\dd{x}\\
&=2 \int_0^{\infty}
\frac{|\log T|}{(1+\ee^{-x})^{2\mu}}\ee^{-\mu x}\dd x\\
&\leq 2 \int_0^{\infty}
\frac{|\log T|}{(1+0)^{2\mu}}\ee^{-\mu x}\dd x\\
&=\frac{2}{\mu}|\log T|.
\end{align*}
For the second term of the integral, we have
\begin{align*}
&\frac{1}{\cos(d/2)}
\int_{-\infty}^{\infty}
 \frac{\log(1+\ee^{-x})}{(1+\ee^{-x})^{\alpha}(1+\ee^{x})^{\beta}}\dd{x}\\
&=\frac{1}{\cos(d/2)}\int_{-\infty}^{0}
 \frac{\ee^{\alpha x}\left\{-x + \log(1+\ee^{x})\right\}}{(1+\ee^{x})^{\alpha+\beta}}\dd{x}
+\frac{1}{\cos(d/2)}\int_{0}^{\infty}
 \frac{\ee^{-\beta x}\log(1+\ee^{-x})}{(1+\ee^{-x})^{\alpha+\beta}}\dd{x}\\
&\leq\frac{1}{\cos(d/2)}\int_{-\infty}^{0}
 \frac{\ee^{\alpha x}\left\{-x+\log(1+\ee^{0})\right\}}{(1+0)^{\alpha+\beta}}\dd{x}
+\frac{1}{\cos(d/2)}\int_{0}^{\infty}
 \frac{\ee^{-\beta x}\log(1+\ee^{-0})}{(1+0)^{\alpha+\beta}}\dd{x}\\
&=\frac{1}{\alpha\cos(d/2)}\left(\log 2 + \frac{1}{\alpha}\right)
+\frac{\log 2}{\beta\cos(d/2)}.
\end{align*}
Furthermore, using $\mu=\min\{\alpha,\beta\}$, we obtain the conclusion.
\end{proof}

Next, we estimate the truncation error.
For this purpose, we use the following result.

\begin{proposition}
\label{prop:SE1-monotone}
Let $\alpha$ and $\beta$ be positive constants.
Let $G_{-}(x)= - x\ee^{\alpha x}$
and $G_{+}(x)= x\ee^{-\beta x}$.
Then, $G_{-}$ monotonically increases for $x\leq -1/\alpha$,
and $G_{+}$ monotonically decreases for $x\geq 1/\beta$.
\end{proposition}

We bound the truncation error as follows.

\begin{lemma}
\label{lem:SE1-truncation}
Let $K$, $\alpha$ and $\beta$ be positive constants.
Assume that $f$ satisfies~\eqref{eq:f-bound-log-algebraic}
for all $z\in (0, T)$.
Let $\mu=\min\{\alpha,\beta\}$,
let $n$ be positive integer,
and
let $M$ and $N$ be selected by~\eqref{eq:def-MN-SE}.
Let $Mh\geq 1/\alpha$ be satisfied.
Then, putting $F(x)=f(\psi_1(x))\psi_1'(x)$, we have
\begin{align*}
\left|
h\sum_{k=-\infty}^{-M-1}F(kh)
+h\sum_{k=N+1}^{\infty}F(kh)
\right|
\leq \frac{KT^{\alpha+\beta-1}}{\mu}
\left\{2|\log T|+ 2\log 2 + \frac{1}{\mu}
+nh
\right\}\ee^{-\mu n h}.
\end{align*}
\end{lemma}
\begin{proof}
From~\eqref{eq:f-bound-log-algebraic}, it holds that
\begin{align*}
 |F(x)|&\leq
\frac{KT^{\alpha+\beta-1}}{(1+\ee^{-x})^{\alpha}(1+\ee^{x})^{\beta}}
\left|\log\left(\frac{T}{1+\ee^{-x}}\right)\right|\\
&\leq
\frac{KT^{\alpha+\beta-1}}{(1+\ee^{-x})^{\alpha}(1+\ee^{x})^{\beta}}
\left\{|\log T| + \log(1 +\ee^{-x})\right\},
\end{align*}
which is also obtained by substituting $y=0$ into~\eqref{eq:bound-F-SE1}.
Using this inequality, we have
\begin{align*}
&\left|
h\sum_{k=-\infty}^{-M-1}F(kh)
+h\sum_{k=N+1}^{\infty}F(kh)
\right|\\
&\leq h\sum_{k=-\infty}^{-M-1}
\frac{KT^{\alpha+\beta-1}}{(1+\ee^{-kh})^{\alpha}(1+\ee^{kh})^{\beta}}
\left\{|\log T| + \log(1+\ee^{-kh})\right\}\\
&\quad + h\sum_{k=N+1}^{\infty}
\frac{KT^{\alpha+\beta-1}}{(1+\ee^{-kh})^{\alpha}(1+\ee^{kh})^{\beta}}
\left\{|\log T| + \log(1+\ee^{-kh})\right\}\\
& =h\sum_{k=-\infty}^{-M-1}
\frac{KT^{\alpha+\beta-1}\ee^{\alpha k h}}{(1+\ee^{kh})^{\alpha+\beta}}
\left\{|\log T| + (-kh) + \log(1+\ee^{kh})\right\}\\
&\quad + h\sum_{k=N+1}^{\infty}
\frac{KT^{\alpha+\beta-1}\ee^{-\beta k h}}{(1+\ee^{-kh})^{\alpha+\beta}}
\left\{|\log T| + \log(1+\ee^{-kh})\right\}\\
&\leq h\sum_{k=-\infty}^{-M-1}
\frac{KT^{\alpha+\beta-1}\ee^{\alpha k h}}{(1+0)^{\alpha+\beta}}
\left\{|\log T| + (-kh) + \log(1+\ee^{0})\right\}\\
&\quad + h\sum_{k=N+1}^{\infty}
\frac{KT^{\alpha+\beta-1}\ee^{-\beta k h}}{(1+ 0)^{\alpha+\beta}}
\left\{|\log T| + \log(1+\ee^{-0})\right\}.
\end{align*}
Using Proposition~\ref{prop:SE1-monotone},
for $Mh\geq 1/\alpha$, we bound the first sum as
\begin{align*}
&h\sum_{k=-\infty}^{-M-1}
KT^{\alpha+\beta-1}
\left\{|\log T| + (-kh) + \log 2\right\}\ee^{\alpha k h}\\
&\leq KT^{\alpha+\beta-1}\int_{-\infty}^{-Mh}
\left\{|\log T| + (-x) + \log 2\right\}\ee^{\alpha x}\dd{x}\\
&=\frac{KT^{\alpha+\beta-1}}{\alpha}
\left(|\log T| + M h + \log 2 + \frac{1}{\alpha} \right)
\ee^{-\alpha M h}\\
&\leq\frac{KT^{\alpha+\beta-1}}{\mu}
\left(|\log T| + M h + \log 2 + \frac{1}{\mu} \right)
\ee^{-\alpha M h},
\end{align*}
where $\mu=\min\{\alpha,\beta\}$ is used
at the last inequality.
Similarly, we bound the second term as
\begin{align*}
 h\sum_{k=N+1}^{\infty}
KT^{\alpha+\beta-1} \left(|\log T| + \log 2\right)
\ee^{-\beta k h}
&\leq KT^{\alpha+\beta-1}\left(|\log T| + \log 2\right)
\int_{Nh}^{\infty}\ee^{-\beta x}\dd{x}\\
& = KT^{\alpha+\beta-1}\left(|\log T| + \log 2\right)
\cdot\frac{1}{\beta}\ee^{-\beta N h}\\
&\leq KT^{\alpha+\beta-1}\left(|\log T| + \log 2\right)
\cdot\frac{1}{\mu}\ee^{-\beta N h}.
\end{align*}
The final task to obtain the conclusion is
showing $Mh\leq nh$,
$\ee^{-\alpha M h}\leq \ee^{-\mu n h}$ and
$\ee^{-\beta N h}\leq \ee^{-\mu n h}$.
Using~\eqref{eq:def-MN-SE}, we have $M \leq n$ because
both $M$ and $n$ are integers that satisfy
\[
 M = \left\lceil\frac{\mu}{\alpha}n\right\rceil
   < \frac{\mu}{\alpha}n + 1 \leq n+1.
\]
Therefore, we have $Mh\leq nh$.
Furthermore, using~\eqref{eq:def-MN-SE},
we have $\ee^{-\alpha M h}\leq \ee^{-\mu n h}$ because
\[
 \alpha M = \alpha \left\lceil\frac{\mu}{\alpha}n\right\rceil
 \geq \alpha \cdot \frac{\mu}{\alpha} n = \mu n.
\]
In the same manner, using~\eqref{eq:def-MN-SE},
we have $\ee^{-\beta N h}\leq \ee^{-\mu n h}$
because $\beta N\geq \mu n$,
which completes the proof.
\end{proof}

We are now in a position to prove Theorem~\ref{thm:new-SE1}.
Note that if $n\geq 1/(2\pi d \mu)$,
then with $h$ selected by~\eqref{eq:standard-h-SE},
it holds that
\begin{equation}
 Mh = \left\lceil\frac{\mu}{\alpha} n\right\rceil
\sqrt{\frac{2\pi d}{\mu n}}
\geq \frac{\mu}{\alpha} n \sqrt{\frac{2\pi d}{\mu n}}
=\frac{1}{\alpha}\sqrt{2\pi d \mu n} \geq \frac{1}{\alpha}.
\label{eq:Mh-bounded-by-1-alpha}
\end{equation}
Therefore,
from Lemmas~\ref{lem:SE1-discretization}
and~\ref{lem:SE1-truncation},
substituting~\eqref{eq:standard-h-SE} into $h$,
we have
\begin{align*}
\left|
\int_{-\infty}^{\infty}F(x)\dd{x}
- h \sum_{k=-M}^N F(kh)
\right| \leq C(n) \sqrt{n}\ee^{-\sqrt{2\pi d \mu n}},
\end{align*}
where
\[
 C(n)
=\frac{K T^{\alpha+\beta-1}}{\mu}
\left[
\frac{4|\log T|\cos(d/2) + 2 l_{\mu}}{\sqrt{n}(1 - \ee^{-\sqrt{2\pi d \mu n}})\cos^{\alpha+\beta+1}(d/2)}
+\frac{2|\log T| + l_{\mu}}{\sqrt{n}}
+\sqrt{\frac{2\pi d}{\mu}}
\right],
\]
where $l_{\mu}=2\log 2 + (1/\mu)$.
Furthermore, $C(n)\leq C(1)$ holds, which completes the proof of
Theorem~\ref{thm:new-SE1}.

\subsubsection{Proof of Theorem~\ref{thm:new-DE1}}

In the case of the DE formula as well,
we estimate both the discretization and truncation errors.
To bound the discretization error,
the following bounds are required.

\begin{lemma}[Okayama et al.~{\cite[Lemma~4.22]{OkaMatsuSugi}}]
\label{lem:bound-DE-1p-exp}
For all real numbers $x$ and $y$ with $|y|<\pi/2$,
it holds that
\begin{align*}
\left|
\frac{1}{1+\ee^{\pi\sinh(x + \ii y)}}
\right|
&\leq \frac{1}{(1+\ee^{\pi\sinh(x)\cos y})\cos((\pi/2)\sin y)},\\
\left|
\frac{1}{1+\ee^{-\pi\sinh(x + \ii y)}}
\right|
&\leq \frac{1}{(1+\ee^{-\pi\sinh(x)\cos y})\cos((\pi/2)\sin y)}.
\end{align*}
\end{lemma}

\begin{lemma}
\label{lem:bound-DE1-log}
Let $T$ be a positive real number.
For all real numbers $x$ and $y$ with $|y|<\pi/2$,
it holds that
\begin{align*}
\left|
\log\left(\frac{T}{1+\ee^{-\pi\sinh(x+\ii y)}}\right)
\right|
&\leq |\log T|
 + \frac{1}{\cos((\pi/2)\sin y)\cos y}\log(1+\ee^{-\pi\sinh(x)\cos y}).
\end{align*}
\end{lemma}
\begin{proof}
First, it holds that
\begin{align*}
\left|
\log\left(\frac{T}{1+\ee^{-\pi\sinh(x + \ii y)}}\right)
\right|
&=   \left|\log T - \log(1 + \ee^{-\pi\sinh(x + \ii y)})\right|\\
&\leq |\log T| + |\log(1 + \ee^{-\pi\sinh(x+\ii y)})|.
\end{align*}
Furthermore, noting
$(\log(1+\ee^{-\pi\sinh\zeta}))'= - \pi\cosh\zeta/(1+\ee^{\pi\sinh\zeta})$
and using Lemma~\ref{lem:bound-DE-1p-exp},
we have
\begin{align*}
 |\log(1+\ee^{-\pi\sinh(x+\ii y)})|
&=
\left|\int_x^{\infty}
\frac{\pi\cosh(t+\ii y)}{1+\ee^{\pi\sinh(t+\ii y)}}\dd{t}\right|\\
&\leq \int_x^{\infty}
\frac{\pi|\cosh(t+\ii y)|}{|1+\ee^{\pi\sinh(t+\ii y)}|}\dd{t}\\
&\leq \int_x^{\infty}
\frac{\pi\cosh t}{(1+\ee^{\pi\sinh(t)\cos y})\cos((\pi/2)\sin y)}\dd{t}\\
&= \frac{1}{\cos((\pi/2)\sin y)\cos y}\log(1+\ee^{-\pi\sinh(x)\cos y}),
\end{align*}
from which we obtain the desired inequality.
\end{proof}

Using these bounds, we show the following lemma.

\begin{lemma}
\label{lem:DE1-discretization}
Let $K$, $\alpha$, $\beta$ and $d$ be positive constants with $d<\pi/2$.
Assume that $f$ is analytic on $\phi_1(\domD_d)$, and
satisfies~\eqref{eq:f-bound-log-algebraic} for all $z\in\phi_1(\domD_d)$.
Let $\mu=\min\{\alpha,\beta\}$. Then, putting
$F(x)=f(\phi_1(x))\phi_1'(x)$, we have~\eqref{eq:discretization-error},
where
\[
\mathcal{N}(F,d)
\leq\frac{2KT^{\alpha+\beta-1}}{\mu\cos^{\alpha+\beta}((\pi/2)\sin d)\cos d}
\left\{2|\log T|+
 \frac{2\log 2 + (1/\mu)}{\cos((\pi/2)\sin d)\cos d}
\right\}.
\]
\end{lemma}
\begin{proof}
It suffices to show that $F\in\mathbf{B}(\domD_d)$.
Because $f(\phi_1(\cdot))$ is analytic on $\domD_d$ and
$\phi_1'$ is analytic on $\domD_{\pi/2}$,
$F$ is analytic on $\domD_d$.
Next, we show~\eqref{eq:F-imag-int-0}.
From~\eqref{eq:f-bound-log-algebraic}
and Lemmas~\ref{lem:bound-DE-1p-exp}
and~\ref{lem:bound-DE1-log}, it holds
for $\zeta=x+\ii y\in\domD_d$ that
\begin{align}
 |F(\zeta)|&\leq
\frac{KT^{\alpha+\beta-1}\pi|\cosh\zeta|}{|1+\ee^{-\pi\sinh\zeta}|^{\alpha}|1+\ee^{\pi\sinh\zeta}|^{\beta}}
\left|\log\left(\frac{T}{1+\ee^{-\pi\sinh\zeta}}\right)\right|\nonumber\\
&\leq
\frac{KT^{\alpha+\beta-1}\pi\cosh x}
{(1+\ee^{-\pi\sinh(x)\cos y})^{\alpha}(1+\ee^{\pi\sinh(x)\cos  y})^{\beta}\cos^{\alpha+\beta}((\pi/2)\sin y)}\nonumber\\
&\quad\cdot
\left\{
|\log T| + \frac{\log(1+\ee^{-\pi\sinh(x)\cos y})}{\cos((\pi/2)\sin y)\cos y}
\right\}\nonumber\\
&\leq \frac{KT^{\alpha+\beta-1}\pi\cosh x}
{(1+\ee^{-\pi\sinh(x)\cos y})^{\mu}(1+\ee^{\pi\sinh(x)\cos  y})^{\mu}\cos^{\alpha+\beta}((\pi/2)\sin y)}\nonumber\\
&\quad\cdot
\left\{
|\log T| + \frac{\log(1+\ee^{-\pi\sinh(x)\cos y})}{\cos((\pi/2)\sin y)\cos y}
\right\},
\label{eq:bound-F-DE1}
\end{align}
where $\mu=\min\{\alpha,\beta\}$ is used
at the last inequality.
Using this inequality~\eqref{eq:bound-F-DE1}, for $x\geq 0$ we have
\begin{align*}
    \int_{-d}^d\left|F(x + \ii y)\right|\dd y
&\leq
\int_{-d}^d
\frac{KT^{\alpha+\beta-1}\pi\cosh x}{(1+\ee^{-\pi\sinh(x)\cos y})^{\mu}(1+\ee^{\pi\sinh(x)\cos y})^{\mu}\cos^{\alpha+\beta}((\pi/2)\sin y)}\\
&\quad\quad\cdot
\left(|\log T| + \frac{\log(1+\ee^{-\pi\sinh(x)\cos y})}{\cos((\pi/2)\sin y)\cos y}\right)\dd{y}\\
&\leq
\frac{KT^{\alpha+\beta-1}\pi\cosh x}{(1+\ee^{-\pi\sinh(x)\cos d})^{\mu}(1+\ee^{\pi\sinh(x)\cos d})^{\mu}\cos^{\alpha+\beta}((\pi/2)\sin d)}\\
&\quad\cdot
\int_{-d}^d
\left(|\log T| + \frac{\log(1+\ee^{-\pi\sinh(0)\cos y})}{\cos((\pi/2)\sin y)\cos y}\right)\dd{y}\\
&\to 0 \cdot
\int_{-d}^d
\left(|\log T| + \frac{\log 2}{\cos((\pi/2)\sin y)\cos y}\right)\dd{y}
\quad (x \to \infty),
\end{align*}
and for $x < 0$,
using $\log(1 + \ee^{-\pi\sinh(x)\cos y})\leq -\pi\sinh(x)\cos(y) + 1$,
we have
\begin{align*}
&     \int_{-d}^d\left|F(x + \ii y)\right|\dd y\\
&\leq\int_{-d}^d\frac{KT^{\alpha+\beta-1}\pi\cosh x\cdot(-\pi\sinh(x)\cos(y) +1)}{(1+\ee^{-\pi\sinh(x)\cos y})^{\mu}(1+\ee^{\pi\sinh(x)\cos y})^{\mu}\cos^{\alpha+\beta}((\pi/2)\sin y)}\\
&\quad\quad\cdot
\left(\frac{|\log T|}{-\pi\sinh(x)\cos(y)+1} + \frac{\log(1+\ee^{-\pi\sinh(x)\cos y})}{\cos((\pi/2)\sin y)(-\pi\sinh(x)\cos y+1)\cos y}\right)\dd{y}\\
&\leq\frac{KT^{\alpha+\beta-1}\pi\cosh x\cdot(-\pi\sinh(x)\cos(0) + 1)}{(1+\ee^{-\pi\sinh(x)\cos d})^{\mu}(1+\ee^{\pi\sinh(x)\cos d})^{\mu}\cos^{\alpha+\beta}((\pi/2)\sin d)}\\
&\quad\cdot
\int_{-d}^d
\left(\frac{|\log T|}{-\pi\sinh(x)\cos(y)+1} + \frac{\log(1+\ee^{-\pi\sinh(x)\cos y})}{\cos((\pi/2)\sin y)(-\pi\sinh(x)\cos(y)+1)\cos y}\right)\dd{y}\\
&\leq\frac{KT^{\alpha+\beta-1}\pi\cosh x\cdot(-\pi\sinh(x) + 1)}{(1+\ee^{-\pi\sinh(x)\cos d})^{\mu}(1+\ee^{\pi\sinh(x)\cos d})^{\mu}\cos^{\alpha+\beta}((\pi/2)\sin d)}\\
&\quad\cdot
\int_{-d}^d
\left(\frac{|\log T|}{-\pi\sinh(0)\cos(y)+1} + \frac{1}{\cos((\pi/2)\sin y)\cos y}\right)\dd{y}\\
&\to 0 \cdot
\int_{-d}^d
\left(|\log T| + \frac{1}{\cos((\pi/2)\sin y)\cos y}\right)\dd{y}
\quad (x \to -\infty),
\end{align*}
which shows~\eqref{eq:F-imag-int-0}.
Finally, we estimate $\mathcal{N}(F,d)$.
Using~\eqref{eq:bound-F-DE1},
we have
\begin{align*}
&\int_{-\infty}^{\infty} \left\{|F(x + \ii y)|+|F(x - \ii y)|\right\}\dd x\\
&\leq\int_{-d}^d
\frac{KT^{\alpha+\beta-1}\pi\cosh x}{(1+\ee^{-\pi\sinh(x)\cos y})^{\mu}(1+\ee^{\pi\sinh(x)\cos y})^{\mu}\cos^{\alpha+\beta}((\pi/2)\sin y)}\\
&\quad\quad\cdot
\left(|\log T| + \frac{\log(1+\ee^{-\pi\sinh(x)\cos y})}{\cos((\pi/2)\sin y)\cos y}\right)\dd{x}\\
&\quad+\int_{-d}^d
\frac{KT^{\alpha+\beta-1}\pi\cosh x}{(1+\ee^{-\pi\sinh(x)\cos(-y)})^{\mu}(1+\ee^{\pi\sinh(x)\cos(-y)})^{\mu}\cos^{\alpha+\beta}((\pi/2)\sin(-y))}\\
&\quad\quad\cdot
\left(|\log T| + \frac{\log(1+\ee^{-\pi\sinh(x)\cos(-y)})}{\cos((\pi/2)\sin(-y))\cos y}\right)\dd{x}\\
&\leq\frac{2KT^{\alpha+\beta-1}}{\cos^{\alpha+\beta}((\pi/2)\sin d)}
\int_{-\infty}^{\infty}
\frac{\pi\cosh x}{(1+\ee^{-\pi\sinh(x)\cos d})^{\mu}(1+\ee^{\pi\sinh(x)\cos d})^{\mu}}\\
&\qquad\qquad\qquad\qquad\qquad\cdot
\left(|\log T| + \frac{\log(1+\ee^{-\pi\sinh(x)\cos d})}{\cos((\pi/2)\sin d)\cos d}\right)\dd{x},
\end{align*}
where $\cos d\leq\cos(\pm y)$ and
$\cos((\pi/2)\sin d)\leq \cos((\pi/2)\sin(\pm y))$ are used,
which hold for $y\in [-d,d]$ (note that $d<\pi/2$).
Therefore, the inequality remains valid
when taking the limit $y\to d - 0$.
For the first term of the integral, we have
\begin{align*}
&\int_{-\infty}^{\infty}
 \frac{|\log T|\pi\cosh x}{(1+\ee^{-\pi\sinh(x)\cos d})^{\mu}(1+\ee^{\pi\sinh(x)\cos d})^{\mu}}\dd{x}\\
&=2 \int_0^{\infty}
 \frac{|\log T|\pi\cosh x}{(1+\ee^{-\pi\sinh(x)\cos d})^{2\mu}}
\ee^{-\pi\mu\sinh(x)\cos d}\dd{x}\\
&\leq 2 \int_0^{\infty}
\frac{|\log T|\pi\cosh x}{(1+0)^{2\mu}}
\ee^{-\pi\mu\sinh(x)\cos d}\dd{x}\\
&=\frac{2}{\mu\cos d}|\log T|.
\end{align*}
For the second term of the integral, we have
\begin{align*}
&\frac{1}{\cos((\pi/2)\sin d)\cos d}
\int_{-\infty}^{\infty}
 \frac{\log(1+\ee^{-\pi\sinh(x)\cos d})\pi\cosh x}{(1+\ee^{-\pi\sinh(x)\cos d})^{\mu}(1+\ee^{\pi\sinh(x)\cos d})^{\mu}}\dd{x}\\
&=\frac{1}{\cos((\pi/2)\sin d)\cos d}\int_{-\infty}^{0}
\frac{\left\{-\pi\sinh(x)\cos d + \log(1+\ee^{\pi\sinh(x)\cos d})\right\}\pi\cosh x}{(1+\ee^{\pi\sinh(x)\cos d})^{2\mu}\ee^{-\pi\mu\sinh(x)\cos d}}\dd{x}\\
&\quad+
\frac{1}{\cos((\pi/2)\sin d)\cos d}\int_{0}^{\infty}
\frac{\log(1+\ee^{-\pi\sinh(x)\cos d})\pi\cosh x}{(1+\ee^{-\pi\mu\sinh(x)\cos d})^{2\mu}\ee^{\pi\mu\sinh(x)\cos d}}\dd{x}\\
&\leq\frac{1}{\cos((\pi/2)\sin d)\cos d}\int_{-\infty}^{0}
\frac{\left\{-\pi\sinh(x)\cos d + \log(1+\ee^{\pi\sinh(0)\cos d})\right\}\pi\cosh x}{(1+0)^{2\mu}\ee^{-\pi\mu\sinh(x)\cos d}}\dd{x}\\
&\quad+
\frac{1}{\cos((\pi/2)\sin d)\cos d}\int_{0}^{\infty}
\frac{\log(1+\ee^{-\pi\sinh(0)\cos d})\pi\cosh x}{(1+0)^{2\mu}\ee^{\pi\mu\sinh(x)\cos d}}\dd{x}\\
&=\frac{1}{\mu\cos((\pi/2)\sin d)\cos^2 d}
\left(\log 2 + \frac{1}{\mu}\right)
+
\frac{\log 2}{\mu\cos((\pi/2)\sin d)\cos^2 d}.
\end{align*}
Thus, we obtain the conclusion.
\end{proof}

Next, we estimate the truncation error.
For this purpose, we use the following result.

\begin{proposition}
\label{prop:DE1-monotone}
Let $\alpha$ and $\beta$ be positive constants.
Let $G_{-}$, $G_{+}$,
$\tilde{G}_{-}$, and $\tilde{G}_{+}$ be defined by
\begin{align*}
 G_{-}(x)&= - \sinh(x)\cosh(x)\ee^{\pi\alpha\sinh x},\\
 G_{+}(x)&=\sinh(x)\cosh(x)\ee^{-\pi\beta\sinh x},\\
\tilde{G}_{-}(x)&=\cosh(x)\ee^{\pi\alpha\sinh x},\\
\tilde{G}_{+}(x)&=\cosh(x)\ee^{-\pi\beta\sinh x}.
\end{align*}
Then, $G_{-}$
and $\tilde{G}_{-}$
monotonically increase for $x\leq -\arsinh(2/(\pi\alpha))$,
and $G_{+}$ and $\tilde{G}_{+}$
monotonically decrease for $x\geq \arsinh(2/(\pi\beta))$.
\end{proposition}
\begin{proof}
From $x\leq -\arsinh(2/(\pi\alpha))$,
$\pi\alpha\sinh x\leq -2$ holds,
from which we have
\[
 1 + \pi\alpha\sinh x\leq -1 < -\frac{\sinh^2 x}{\cosh^2 x}.
\]
From the inequality,
$\sinh^2 x + (1 + \pi\alpha\sinh x)\cosh^2 x < 0$ holds,
from which we have
\[
 G_{-}'(x) = -\left\{
\sinh^2 x + (1 + \pi\alpha\sinh x)\cosh^2 x
\right\}\ee^{\pi\alpha\sinh x}
> 0.
\]
Thus, the claim on the function $G_{-}$ follows.
Similarly, from $x\geq \arsinh(2/(\pi\beta))$,
$\pi\beta\sinh x\geq 2$ holds,
from which we have
\[
 1 - \pi\beta\sinh x\leq -1 < -\frac{\sinh^2 x}{\cosh^2 x}.
\]
From the inequality,
$\sinh^2 x + (1 - \pi\beta\sinh x)\cosh^2 x < 0$ holds,
from which we have
\[
 G_{+}'(x) = \left\{
\sinh^2 x + (1 - \pi\beta\sinh x)\cosh^2 x
\right\}\ee^{-\pi\beta\sinh x}
< 0.
\]
Thus, the claim on the function $G_{+}$ follows.

Next, we consider $\tilde{G}_{-}$ and $\tilde{G}_{+}$.
If $\alpha\geq 1/(2\pi)$ and $\beta\geq 1/(2\pi)$,
then $\tilde{G}_{-}$ monotonically increases for $x\leq 0$
and $\tilde{G}_{+}$ monotonically decreases for $x\geq 0$.
Therefore, in this case, it is also true that
$\tilde{G}_{-}$ monotonically increases for $x\leq -\arsinh(2/(\pi\alpha))$
and $\tilde{G}_{+}$ monotonically decreases for
$x\geq \arsinh(2/(\pi\beta))$.
If $\alpha< 1/(2\pi)$ and $\beta< 1/(2\pi)$,
then $\tilde{G}_{-}$ monotonically increases for $x\leq -x_{\alpha}$
and $\tilde{G}_{+}$ monotonically decreases for $x\geq x_{\beta}$,
where $x_{\gamma}$ is defined for $\gamma\in (0,1/(2\pi))$ by
\[
 x_{\gamma} =\arsinh\left(\frac{1+\sqrt{1-(2\pi\gamma)^2}}{2\pi\gamma}\right).
\]
Note that
\[
 x_{\gamma}\leq
 \arsinh\left(\frac{1+\sqrt{1 - 0^2}}{2\pi\gamma}\right)
=\arsinh\left(\frac{1}{\pi\gamma}\right)
<\arsinh\left(\frac{2}{\pi\gamma}\right).
\]
Therefore, in this case as well,
$\tilde{G}_{-}$ monotonically increases for $x\leq -\arsinh(2/(\pi\alpha))$
and $\tilde{G}_{+}$ monotonically decreases for
$x\geq \arsinh(2/(\pi\beta))$.
Thus, the claim on the functions $\tilde{G}_{-}$ and $\tilde{G}_{+}$ follows.
\end{proof}

We bound the truncation error as follows.

\begin{lemma}
\label{lem:DE1-truncation}
Let $K$, $\alpha$ and $\beta$ be positive constants.
Assume that $f$ satisfies~\eqref{eq:f-bound-log-algebraic}
for all $z\in (0, T)$.
Let $\mu=\min\{\alpha,\beta\}$,
let $n$ be positive integer,
and
let $M$ and $N$ be selected by~\eqref{eq:def-MN-DE}.
Let $Mh\geq \arsinh(2/(\pi\alpha))$ and
$Nh\geq \arsinh(2/(\pi\beta))$ be satisfied.
Then, putting $F(x)=f(\phi_1(x))\phi_1'(x)$, we have
\begin{align*}
&\left|
h\sum_{k=-\infty}^{-M-1}F(kh)
+h\sum_{k=N+1}^{\infty}F(kh)
\right|\\
&\leq \frac{KT^{\alpha+\beta-1}}{\mu}
\left\{2|\log T|+ 2\log 2 + \frac{1}{\mu}
+\pi\sinh(Mh)
\right\}\ee^{-\pi \mu q(2dn/\mu)},
\end{align*}
where $q(x)=x/\arsinh x$.
\end{lemma}
\begin{proof}
From~\eqref{eq:f-bound-log-algebraic}, it holds that
\begin{align*}
 |F(x)|&\leq
\frac{KT^{\alpha+\beta-1}\pi\cosh x}{(1+\ee^{-\pi\sinh x})^{\alpha}(1+\ee^{\pi\sinh x})^{\beta}}
\left|\log\left(\frac{T}{1+\ee^{-\pi\sinh x}}\right)\right|\\
&\leq
\frac{KT^{\alpha+\beta-1}\pi\cosh x}{(1+\ee^{-\pi\sinh x})^{\alpha}(1+\ee^{\pi\sinh x})^{\beta}}
\left\{|\log T| + \log(1 +\ee^{-\pi\sinh x})\right\}.
\end{align*}
Using this inequality, we have
\begin{align*}
&\left|
h\sum_{k=-\infty}^{-M-1}F(kh)
+h\sum_{k=N+1}^{\infty}F(kh)
\right|\\
&\leq h\sum_{k=-\infty}^{-M-1}
\frac{KT^{\alpha+\beta-1}\pi\cosh(kh)}{(1+\ee^{-\pi\sinh(kh)})^{\alpha}(1+\ee^{\pi\sinh(kh)})^{\beta}}
\left\{|\log T| + \log(1+\ee^{-\pi\sinh(kh)})\right\}\\
&\quad + h\sum_{k=N+1}^{\infty}
\frac{KT^{\alpha+\beta-1}\pi\cosh(kh)}{(1+\ee^{-\pi\sinh(kh)})^{\alpha}(1+\ee^{\pi\sinh(kh)})^{\beta}}
\left\{|\log T| + \log(1+\ee^{-\pi\sinh(kh)})\right\}\\
& =h\sum_{k=-\infty}^{-M-1}
\frac{KT^{\alpha+\beta-1}\pi\cosh(kh)\left\{|\log T| + (-\pi\sinh(kh)) + \log(1+\ee^{\pi\sinh(kh)})\right\}}{(1+\ee^{\pi\sinh(kh)})^{\alpha+\beta}\ee^{-\pi\alpha\sinh(k h)}}
\\
&\quad + h\sum_{k=N+1}^{\infty}
\frac{KT^{\alpha+\beta-1}\pi\cosh(kh)\left\{|\log T| + \log(1+\ee^{-\pi\sinh(kh)})\right\}}{(1+\ee^{-\pi\sinh(kh)})^{\alpha+\beta}\ee^{\pi\beta\sinh(k h)}}
\\
&\leq h\sum_{k=-\infty}^{-M-1}
\frac{KT^{\alpha+\beta-1}\pi\cosh(kh)\left\{|\log T| + (-\pi\sinh(kh)) + \log(1+\ee^{\pi\sinh(0)})\right\}}{(1+0)^{\alpha+\beta}\ee^{-\pi\alpha\sinh(k h)}}
\\
&\quad + h\sum_{k=N+1}^{\infty}
\frac{KT^{\alpha+\beta-1}\pi\cosh(kh)\left\{|\log T| + \log(1+\ee^{-\pi\sinh(0)})\right\}}{(1+0)^{\alpha+\beta}\ee^{\pi\beta\sinh(k h)}}.
\end{align*}
Using Proposition~\ref{prop:DE1-monotone},
for $Mh\geq \arsinh(2/(\pi\alpha))$, we bound the first sum as
\begin{align*}
&h\sum_{k=-\infty}^{-M-1}
KT^{\alpha+\beta-1}\pi\cosh(kh)\left\{|\log T| + (-\pi\sinh(kh)) + \log 2)\right\}\ee^{\pi\alpha\sinh(k h)}
\\
&\leq KT^{\alpha+\beta-1}\int_{-\infty}^{-Mh}\pi\cosh(x)
\left\{|\log T| + (-\pi\sinh x) + \log 2)\right\}\ee^{\pi\alpha\sinh x}\dd{x}\\
&=\frac{KT^{\alpha+\beta-1}}{\alpha}
\left(|\log T| + \pi\sinh(Mh) + \log 2 + \frac{1}{\alpha} \right)
\ee^{-\pi\alpha\sinh(M h)}\\
&\leq\frac{KT^{\alpha+\beta-1}}{\mu}
\left(|\log T| + \pi\sinh(Mh) + \log 2 + \frac{1}{\mu} \right)
\ee^{-\pi\alpha\sinh(M h)},
\end{align*}
where $\mu=\min\{\alpha,\beta\}$ is used
at the last inequality.
Similarly, for $Nh\geq \arsinh(1/(2\beta))$,
we bound the second term as
\begin{align*}
& h\sum_{k=N+1}^{\infty}
KT^{\alpha+\beta-1}\pi\cosh(kh) \left(|\log T| + \log 2\right)
\ee^{-\pi\beta\sinh(k h)}\\
&\leq KT^{\alpha+\beta-1}\left(|\log T| + \log 2\right)
\int_{Nh}^{\infty}\pi\cosh x \ee^{-\pi\beta\sinh x}\dd{x}\\
& = KT^{\alpha+\beta-1}\left(|\log T| + \log 2\right)
\cdot\frac{1}{\beta}\ee^{-\pi\beta\sinh(N h)}\\
&\leq KT^{\alpha+\beta-1}\left(|\log T| + \log 2\right)
\cdot\frac{1}{\mu}\ee^{-\pi\beta\sinh(N h)}.
\end{align*}
The final task to obtain the conclusion is
showing $\ee^{-\pi\alpha\sinh(M h)}\leq \ee^{-\pi\mu q(2 d n/\mu)}$
and $\ee^{-\pi\beta\sinh(N h)}\leq \ee^{-\pi\mu q(2 d n/\mu)}$.
Using~\eqref{eq:def-MN-DE},
we have $\ee^{-\pi\alpha\sinh(M h)}\leq \ee^{-\pi\mu q(2 d n/\mu)}$
because
\begin{align*}
 \alpha\sinh(Mh)
&=\alpha\sinh\left(\left\lceil
\frac{1}{h}\arsinh\left(\frac{\mu}{\alpha}q\left(\frac{2dn}{\mu}\right)\right)
\right\rceil h
\right)\\
&\geq\alpha\sinh\left(
\frac{1}{h}\arsinh\left(\frac{\mu}{\alpha}q\left(\frac{2dn}{\mu}\right)\right)
h
\right)\\
&=\alpha\left(\frac{\mu}{\alpha}q\left(\frac{2dn}{\mu}\right)\right)\\
&=\mu q\left(\frac{2dn}{\mu}\right).
\end{align*}
In the same manner, using~\eqref{eq:def-MN-DE},
we have $\ee^{-\pi\beta\sinh(N h)}\leq \ee^{-\pi\mu q(2 d n/\mu)}$
because $\beta\sinh(Nh)\geq \mu q(2 d n/\mu)$,
which completes the proof.
\end{proof}

We are now in a position to prove Theorem~\ref{thm:new-DE1}.
Note that if $h\leq \pi d$,
then with $h$ selected by~\eqref{eq:improve-h-DE},
it holds that
\begin{equation}
 Mh
\geq \arsinh\left(\frac{\mu}{\alpha}q\left(\frac{2 d n}{\mu}\right)\right)
=\arsinh\left(\frac{2d}{\alpha h}\right)\geq
\arsinh\left(\frac{2}{\pi\alpha}\right).
\label{eq:Mh-bounded-by-arsinh}
\end{equation}
Similarly, it holds that
\begin{equation}
 Nh\geq \arsinh\left(\frac{\mu}{\beta}q\left(\frac{2 d n}{\mu}\right)\right)
=\arsinh\left(\frac{2d}{\beta h}\right)\geq
\arsinh\left(\frac{2}{\pi\beta}\right).
\label{eq:Nh-bounded-by-arsinh}
\end{equation}
Furthermore, if $n\geq \mu\sinh(1)/(2d)$,
i.e., if $\arsinh(2dn/\mu)\geq 1$,
then with $h$ selected by~\eqref{eq:improve-h-DE},
it holds that
\begin{align}
M &= \left\lceil
\frac{1}{h}\arsinh\left(\frac{2 d n/\alpha}{\arsinh(2 d n/\mu)}\right)
\right\rceil\nonumber\\
&=\left\lceil
\frac{n}{\arsinh(2 d n/\mu)}\arsinh\left(\frac{2 d n/\alpha}{\arsinh(2 d n/\mu)}\right)
\right\rceil\nonumber\\
&\leq\left\lceil
\frac{n}{\arsinh(2 d n/\mu)}\arsinh\left(\frac{2 d n/\alpha}{1}\right)
\right\rceil\nonumber\\
&\leq\left\lceil
\frac{n}{\arsinh(2 d n/\mu)}\arsinh\left(2 d n/\mu\right)
\right\rceil\nonumber\\
&=n,\label{eq:M-leq-n}
\end{align}
from which the truncation error
in Lemma~\ref{lem:DE1-truncation} is further bounded as
\begin{align*}
&\frac{KT^{\alpha+\beta-1}}{\mu}
\left\{2|\log T|+ 2\log 2 + \frac{1}{\mu}
+\pi\sinh(Mh)
\right\}\ee^{-\pi \mu q(2dn/\mu)}\\
&\leq
 \frac{KT^{\alpha+\beta-1}}{\mu}
\left\{2|\log T|+ 2\log 2 + \frac{1}{\mu}
+\pi\sinh(nh)
\right\}\ee^{-\pi \mu q(2dn/\mu)}.
\end{align*}
Therefore,
from Lemmas~\ref{lem:DE1-discretization}
and~\ref{lem:DE1-truncation},
substituting~\eqref{eq:improve-h-DE} into $h$,
we have
\begin{align*}
\left|
\int_{-\infty}^{\infty}F(x)\dd{x}
- h \sum_{k=-M}^N F(kh)
\right| \leq C(n) n\ee^{-\pi\mu q(2 d n/\mu)},
\end{align*}
where
\[
 C(n)
=\frac{K T^{\alpha+\beta-1}}{\mu}
\left[
\frac{c_d^{\alpha+\beta}\left(4|\log T|\cos d + 2 l_{\mu} c_d\right)}{n(1 - \ee^{-\pi\mu q(2dn/\mu)})\cos^{2} d}
+\frac{2|\log T| + l_{\mu}}{n}
+\frac{2\pi d}{\mu}
\right],
\]
where $c_d = 1/\cos((\pi/2)\sin d)$ and $l_{\mu}=2\log 2 + (1/\mu)$.
Furthermore, $C(n)\leq C(1)$ holds, which completes the proof of
Theorem~\ref{thm:new-DE1}.

\subsection{In the case of the semi-infinite interval}
\label{sec:proof-semi-infinite}

Here, we prove theorems in the case of the semi-infinite interval,
i.e., Theorems~\ref{thm:new-SE2}--\ref{thm:new-DE3}.
The strategy to analyze the error is the same as above,
i.e., we estimate both the discretization and truncation errors.

\subsubsection{Proof of Theorem~\ref{thm:new-SE2}}

For Theorem~\ref{thm:new-SE2},
we bound the discretization error as follows.

\begin{lemma}
\label{lem:SE2-discretization}
Let $K$, $\alpha$, $\beta$ and $d$ be positive constants with $d<\pi/2$.
Assume that $f$ is analytic on $\psi_2(\domD_d)$, and
satisfies~\eqref{eq:f-bound-log-algebraic-2} for all $z\in\psi_2(\domD_d)$.
Let $\mu=\min\{\alpha,\beta\}$. Then, putting
$F(x)=f(\psi_2(x))\psi_2'(x)$, we have~\eqref{eq:discretization-error},
where
\[
\mathcal{N}(F,d)
\leq \frac{4K}{\mu\cos^{(\alpha+\beta)/2}(d)}\left(d + \frac{1}{\mu}\right).
\]
\end{lemma}
\begin{proof}
It suffices to show that $F\in\mathbf{B}(\domD_d)$.
Because $f(\psi_2(\cdot))$ is analytic on $\domD_d$ and
$\psi_2'$ is analytic on the whole complex plane,
$F$ is analytic on $\domD_d$.
Next, we show~\eqref{eq:F-imag-int-0}.
From~\eqref{eq:f-bound-log-algebraic-2}
and Lemma~\ref{lem:bound-SE-1p-exp}, it holds
for $\zeta=x+\ii y\in \domD_d$ that
\begin{align}
 |F(\zeta)|&\leq
K\frac{|\zeta|}{|1+\ee^{-2\zeta}|^{\alpha/2}|1+\ee^{2\zeta}|^{\beta/2}}
\nonumber\\
&\leq K\frac{\sqrt{x^2+y^2}}{(1+\ee^{-2x})^{\alpha/2}\cos^{\alpha/2}(y)
(1+\ee^{2x})^{\beta/2}\cos^{\beta/2}(y)}.
\label{eq:bound-F-SE2}
\end{align}
Using this inequality, we have
\begin{align*}
&    \int_{-d}^d\left|F(x + \ii y)\right|\dd y\\
&\leq\frac{K\sqrt{x^2+1}}{(1+\ee^{-2x})^{\alpha/2}(1+\ee^{2x})^{\beta/2}}
\int_{-d}^d
\frac{\sqrt{\{x^2/(x^2+1)\} + \{y^2/(x^2+1)\}}}{\cos^{(\alpha+\beta)/2}(y)}
 \dd{y}\\
& < \frac{K\sqrt{x^2+1}}{(1+\ee^{-2x})^{\alpha/2}(1+\ee^{2x})^{\beta/2}}
\int_{-d}^d
\frac{\sqrt{1 + \{y^2/(0^2+1)\}}}{\cos^{(\alpha+\beta)/2}(y)}
 \dd{y}\\
&\to 0 \cdot
\int_{-d}^d
\frac{\sqrt{1 + y^2}}{\cos^{(\alpha+\beta)/2}(y)}
 \dd{y}
\quad (x \to \pm\infty),
\end{align*}
which shows~\eqref{eq:F-imag-int-0}.
Finally, we estimate $\mathcal{N}(F,d)$.
Using~\eqref{eq:bound-F-SE2},
we have
\begin{align*}
&\int_{-\infty}^{\infty} \left\{|F(x + \ii y)|+|F(x - \ii y)|\right\}\dd x\\
&\leq\frac{K}{\cos^{(\alpha+\beta)/2}(y)}\int_{-\infty}^{\infty}
\frac{\sqrt{x^2+y^2}}{(1+\ee^{-2x})^{\alpha/2}
(1+\ee^{2x})^{\beta/2}}
\dd{x}\\
&\quad +
\frac{K}{\cos^{(\alpha+\beta)/2}(-y)}\int_{-\infty}^{\infty}
\frac{\sqrt{x^2+(-y)^2}}{(1+\ee^{-2x})^{\alpha/2}
(1+\ee^{2x})^{\beta/2}}
\dd{x}\\
&\leq\frac{2K}{\cos^{(\alpha+\beta)/2}(d)}\int_{-\infty}^{\infty}
\frac{\sqrt{x^2+d^2}}{(1+\ee^{-2x})^{\alpha/2}
(1+\ee^{2x})^{\beta/2}}
\dd{x},
\end{align*}
which holds for $y\in [-d,d]$ (note that $d<\pi/2$).
Therefore, the inequality remains valid
when taking the limit $y\to d - 0$.
Using $\mu=\min\{\alpha,\beta\}$,
we bound the integral as
\begin{align*}
\int_{-\infty}^{\infty}
\frac{\sqrt{x^2+d^2}}{(1+\ee^{-2x})^{\alpha/2}
(1+\ee^{2x})^{\beta/2}}
\dd{x}
&\leq \int_{-\infty}^{\infty}
\frac{\sqrt{x^2+d^2}}{(1+\ee^{-2x})^{\mu/2}
(1+\ee^{2x})^{\mu/2}}
\dd{x}\\
&=2 \int_0^{\infty}
\frac{\sqrt{x^2+d^2}}{(1+\ee^{-2x})^{\mu}}\ee^{-\mu x}
\dd{x}\\
&\leq 2 \int_0^{\infty}
\frac{\sqrt{x^2+d^2}}{(1+0)^{\mu}}\ee^{-\mu x}\dd x.
\end{align*}
Furthermore, using integration by parts, we have
\begin{align*}
2 \int_0^{\infty}
\sqrt{x^2+d^2}\ee^{-\mu x}\dd x
&=2 \int_0^{\infty}\sqrt{x^2+d^2}
\left(\frac{\ee^{-\mu x}}{-\mu}\right)'\dd x\\
&=\left[
2\sqrt{x^2+d^2}\left(\frac{\ee^{-\mu x}}{-\mu}\right)
\right]_{x=0}^{x=\infty}
+\frac{2}{\mu}\int_0^{\infty}
\frac{x}{\sqrt{x^2+d^2}}\ee^{-\mu x}\dd{x}\\
&\leq\frac{2d}{\mu}
+\frac{2}{\mu}\int_0^{\infty}\frac{x}{\sqrt{x^2 + 0}}\ee^{-\mu x}\dd{x}\\
&=\frac{2d}{\mu}+\frac{2}{\mu^2}.
\end{align*}
Thus, we obtain the conclusion.
\end{proof}

Next, we bound the truncation error as follows.

\begin{lemma}
\label{lem:SE2-truncation}
Let $K$, $\alpha$ and $\beta$ be positive constants.
Assume that $f$ satisfies~\eqref{eq:f-bound-log-algebraic-2}
for all $z\in (0, \infty)$.
Let $\mu=\min\{\alpha,\beta\}$,
let $n$ be positive integer,
and
let $M$ and $N$ be selected by~\eqref{eq:def-MN-SE}.
Let $Mh\geq 1/\alpha$ and $Nh\geq 1/\beta$ be satisfied.
Then, putting $F(x)=f(\psi_2(x))\psi_2'(x)$, we have
\begin{align*}
\left|
h\sum_{k=-\infty}^{-M-1}F(kh)
+h\sum_{k=N+1}^{\infty}F(kh)
\right|
\leq \frac{2K}{\mu}\left(nh+\frac{1}{\mu}\right)
\ee^{-\mu n h}.
\end{align*}
\end{lemma}
\begin{proof}
From~\eqref{eq:f-bound-log-algebraic-2}, it holds that
\begin{align*}
 |F(x)|\leq K \frac{|\ee^{x}|^{\alpha-1}}{|1+\ee^{2x}|^{(\alpha+\beta)/2}}
|\log\ee^{x}||\ee^x|
=K
\frac{|x|}{(1+\ee^{-2x})^{\alpha/2}(1+\ee^{2x})^{\beta/2}},
\end{align*}
which is also obtained by substituting $y=0$ into~\eqref{eq:bound-F-SE2}.
Using this inequality, we have
\begin{align*}
&\left|
h\sum_{k=-\infty}^{-M-1}F(kh)
+h\sum_{k=N+1}^{\infty}F(kh)
\right|\\
&\leq h\sum_{k=-\infty}^{-M-1}
\frac{K |kh|}{(1+\ee^{-2kh})^{\alpha/2}(1+\ee^{2kh})^{\beta/2}}
+ h\sum_{k=N+1}^{\infty}
\frac{K |kh|}{(1+\ee^{-2kh})^{\alpha/2}(1+\ee^{2kh})^{\beta/2}}\\
& =h\sum_{k=-\infty}^{-M-1}
\frac{K(-kh)\ee^{\alpha k h}}{(1+\ee^{2kh})^{(\alpha+\beta)/2}}
 + h\sum_{k=N+1}^{\infty}
\frac{Kkh\ee^{-\beta k h}}{(1+\ee^{-2kh})^{(\alpha+\beta)/2}}\\
&\leq h\sum_{k=-\infty}^{-M-1}
\frac{K(-kh)\ee^{\alpha k h}}{(1+0)^{(\alpha+\beta)/2}}
 + h\sum_{k=N+1}^{\infty}
\frac{Kkh\ee^{-\beta k h}}{(1+0)^{(\alpha+\beta)/2}}.
\end{align*}
Using Proposition~\ref{prop:SE1-monotone},
for $Mh\geq 1/\alpha$ and $Nh\geq 1/\beta$,
we further bound the sums as
\begin{align*}
&h\sum_{k=-\infty}^{-M-1} K(-kh)\ee^{\alpha k h}
 + h\sum_{k=N+1}^{\infty} Kkh\ee^{-\beta k h}\\
&\leq\int_{-\infty}^{-Mh} K(-x)\ee^{\alpha x}\dd{x}
+\int_{Nh}^{\infty} K x\ee^{-\beta x}\dd{x}\\
&=\frac{K}{\alpha}\left(Mh + \frac{1}{\alpha}\right)\ee^{-\alpha Mh}
+\frac{K}{\beta}\left(Nh + \frac{1}{\beta}\right)\ee^{-\beta Nh}\\
&\leq\frac{K}{\mu}\left(Mh + \frac{1}{\mu}\right)\ee^{-\alpha Mh}
+\frac{K}{\mu}\left(Nh + \frac{1}{\mu}\right)\ee^{-\beta Nh}.
\end{align*}
where $\mu=\min\{\alpha,\beta\}$ is used
at the last inequality.
Finally, using~\eqref{eq:def-MN-SE},
we have $Mh\leq nh$ (because $M\leq n$),
$Nh\leq nh$ (because $N\leq n$),
$\ee^{-\alpha M h}\leq \ee^{-\mu n h}$
(because $\alpha M\geq \mu n$)
and $\ee^{-\beta N h}\leq \ee^{-\mu n h}$ (because $\beta N\geq \mu n$),
from which we obtain the conclusion.
\end{proof}

We are now in a position to prove Theorem~\ref{thm:new-SE2}.
Note that if $n\geq 1/(2\pi d \mu)$,
then with $h$ selected by~\eqref{eq:standard-h-SE},
the inequality~\eqref{eq:Mh-bounded-by-1-alpha} holds,
and it also holds that
\begin{equation}
 Nh = \left\lceil\frac{\mu}{\beta} n\right\rceil
\sqrt{\frac{2\pi d}{\mu n}}
\geq \frac{\mu}{\beta} n \sqrt{\frac{2\pi d}{\mu n}}
=\frac{1}{\beta}\sqrt{2\pi d \mu n} \geq \frac{1}{\beta}.
\label{eq:Nh-bounded-by-1-beta}
\end{equation}
Therefore,
from Lemmas~\ref{lem:SE2-discretization}
and~\ref{lem:SE2-truncation},
substituting~\eqref{eq:standard-h-SE} into $h$,
we have
\begin{align*}
\left|
\int_{-\infty}^{\infty}F(x)\dd{x}
- h \sum_{k=-M}^N F(kh)
\right| \leq C(n) \sqrt{n}\ee^{-\sqrt{2\pi d \mu n}},
\end{align*}
where
\[
 C(n)
=\frac{2 K}{\mu^2}
\left[
\frac{2(\mu d + 1)}{\sqrt{n}(1 - \ee^{-\sqrt{2\pi d \mu n}})\cos^{(\alpha+\beta)/2}(d)}
+ \sqrt{2\pi d \mu} + \frac{1}{\sqrt{n}}
\right].
\]
Furthermore, $C(n)\leq C(1)$ holds, which completes the proof of
Theorem~\ref{thm:new-SE2}.

\subsubsection{Proof of Theorem~\ref{thm:new-DE2}}

For Theorem~\ref{thm:new-DE2},
we estimate the discretization error as follows.

\begin{lemma}
\label{lem:DE2-discretization}
Let $K$, $\alpha$, $\beta$ and $d$ be positive constants with $d<\pi/2$.
Assume that $f$ is analytic on $\phi_2(\domD_d)$, and
satisfies~\eqref{eq:f-bound-log-algebraic-2} for all $z\in\phi_2(\domD_d)$.
Let $\mu=\min\{\alpha,\beta\}$. Then, putting
$F(x)=f(\phi_2(x))\phi_2'(x)$, we have~\eqref{eq:discretization-error},
where
\[
\mathcal{N}(F,d)
\leq \frac{2K}{\mu\cos^{(\alpha+\beta)/2}((\pi/2)\sin d)\cos d}
\left(\pi + \frac{2}{\mu\cos d}\right).
\]
\end{lemma}
\begin{proof}
It suffices to show that $F\in\mathbf{B}(\domD_d)$.
Because $f(\phi_2(\cdot))$ is analytic on $\domD_d$ and
$\phi_2'$ is analytic on the whole complex plane,
$F$ is analytic on $\domD_d$.
Next, we show~\eqref{eq:F-imag-int-0}.
From~\eqref{eq:f-bound-log-algebraic-2}
and Lemma~\ref{lem:bound-DE-1p-exp}, it holds
for $\zeta=x+\ii y\in \domD_d$ that
\begin{align}
 |F(\zeta)|&\leq
K\frac{|(\pi/2)\sinh\zeta||(\pi/2)\cosh\zeta|}{|1+\ee^{-\pi\sinh\zeta}|^{\alpha/2}|1+\ee^{\pi\sinh\zeta}|^{\beta/2}}
\nonumber\\
&\leq K\frac{(\pi/2)^2\cosh^2 x}{(1+\ee^{-\pi\sinh(x)\cos y})^{\alpha/2}(1+\ee^{\pi\sinh(x)\cos y})^{\beta/2}\cos^{(\alpha+\beta)/2}((\pi/2)\sin y)}\nonumber\\
&\leq K\frac{(\pi/2)^2\cosh^2 x}{(1+\ee^{-\pi\sinh(x)\cos y})^{\mu/2}(1+\ee^{\pi\sinh(x)\cos y})^{\mu/2}\cos^{(\alpha+\beta)/2}((\pi/2)\sin y)},
\label{eq:bound-F-DE2}
\end{align}
where $\mu=\min\{\alpha,\beta\}$ is used
at the last inequality.
Using this inequality~\eqref{eq:bound-F-DE2}, we have
\begin{align*}
& \int_{-d}^d\left|F(x + \ii y)\right|\dd y\\
&\leq\int_{-d}^d
\frac{K(\pi/2)^2\cosh^2 x}{(1+\ee^{-\pi\sinh(x)\cos y})^{\mu/2}(1+\ee^{\pi\sinh(x)\cos y})^{\mu/2}\cos^{(\alpha+\beta)/2}((\pi/2)\sin y)}\dd{y}\\
&\leq\frac{K(\pi/2)^2\cosh^2 x}{(1+\ee^{-\pi\sinh(x)\cos d})^{\mu/2}(1+\ee^{\pi\sinh(x)\cos d})^{\mu/2}\cos^{(\alpha+\beta)/2}((\pi/2)\sin d)}
\int_{-d}^d
 \dd{y}\\
&\to 0 \cdot 2d
\quad (x \to \pm\infty),
\end{align*}
which shows~\eqref{eq:F-imag-int-0}.
Finally, we estimate $\mathcal{N}(F,d)$.
Using~\eqref{eq:bound-F-DE2},
we have
\begin{align*}
&\int_{-\infty}^{\infty} \left\{|F(x + \ii y)|+|F(x - \ii y)|\right\}\dd x\\
&\leq
\frac{K}{\cos^{(\alpha+\beta)/2}((\pi/2)\sin y)}
\left\{
\int_{-\infty}^{\infty}
\frac{(\pi/2)^2\cosh^2 x}{(1+\ee^{-\pi\sinh(x)\cos y})^{\mu/2}(1+\ee^{\pi\sinh(x)\cos y})^{\mu/2}}
\dd{x}\right.\\
&\qquad\qquad\qquad\qquad\qquad +\left.
\int_{-\infty}^{\infty}
\frac{(\pi/2)^2\cosh^2 x}{(1+\ee^{-\pi\sinh(x)\cos(-y)})^{\mu/2}(1+\ee^{\pi\sinh(x)\cos(-y)})^{\mu/2}}
\dd{x}\right\}\\
&\leq\frac{2K}{\cos^{(\alpha+\beta)/2}((\pi/2)\sin d)}\int_{-\infty}^{\infty}
\frac{(\pi/2)^2\cosh^2 x}{(1+\ee^{-\pi\sinh(x)\cos d})^{\mu/2}(1+\ee^{\pi\sinh(x)\cos d})^{\mu/2}}
\dd{x},
\end{align*}
which holds for $y\in [-d,d]$ (note that $d<\pi/2$).
Therefore, the inequality remains valid
when taking the limit $y\to d - 0$.
We bound the integral as
\begin{align*}
&\int_{-\infty}^{\infty}
\frac{(\pi/2)^2\cosh^2 x}{(1+\ee^{-\pi\sinh(x)\cos d})^{\mu/2}(1+\ee^{\pi\sinh(x)\cos d})^{\mu/2}}
\dd{x}\\
&=2 \int_0^{\infty}
\frac{(\pi/2)^2\cosh^2 x}{(1+\ee^{-\pi\sinh(x)\cos d})^{\mu}}
\ee^{-(\pi/2)\mu\sinh(x)\cos d}\dd{x}\\
&\leq 2 \int_0^{\infty}
\frac{(\pi/2)^2\cosh^2 x}{(1+0)^{\mu}}
\ee^{-(\pi/2)\mu\sinh(x)\cos d}\dd{x}.
\end{align*}
Furthermore, using integration by parts, we have
\begin{align*}
&2 \int_0^{\infty}
\left(\frac{\pi}{2}\cosh x\right)^2
\ee^{-(\pi/2)\mu\sinh(x)\cos d}\dd{x}\\
&=-\int_0^{\infty}
\frac{\pi\cosh x}{\mu \cos d}
\left(\ee^{-(\pi/2)\mu\sinh(x)\cos d}\right)'\dd{x}\\
&=\left[-
\frac{\pi\cosh x}{\mu \cos d}
\left(\ee^{-(\pi/2)\mu\sinh(x)\cos d}\right)
\right]_{x=0}^{x=\infty}
+\int_0^{\infty}
\frac{\pi\sinh x}{\mu\cos d}
\ee^{-(\pi/2)\mu\sinh(x)\cos d}\dd{x}\\
&\leq\frac{\pi}{\mu\cos d}
+\int_0^{\infty}
\frac{\pi\cosh x}{\mu\cos d}
\ee^{-(\pi/2)\mu\sinh(x)\cos d}\dd{x}\\
&=\frac{\pi}{\mu\cos d}+\frac{2}{\mu^2\cos^2 d}.
\end{align*}
Thus, we obtain the conclusion.
\end{proof}

Next, we bound the truncation error as follows.

\begin{lemma}
\label{lem:DE2-truncation}
Let $K$, $\alpha$ and $\beta$ be positive constants.
Assume that $f$ satisfies~\eqref{eq:f-bound-log-algebraic-2}
for all $z\in (0, \infty)$.
Let $\mu=\min\{\alpha,\beta\}$,
let $n$ be positive integer,
and
let $M$ and $N$ be selected by~\eqref{eq:def-MN-DE2}.
Let $Mh\geq\arsinh(4/(\pi\alpha))$ and $Nh\geq\arsinh(4/(\pi\beta))$
be satisfied.
Then, putting $F(x)=f(\phi_2(x))\phi_2'(x)$, we have
\begin{align*}
&\left|
h\sum_{k=-\infty}^{-M-1}F(kh)
+h\sum_{k=N+1}^{\infty}F(kh)
\right|\\
&\leq \frac{K}{2\mu}\left(\pi\sinh(Mh)+\pi\sinh(Nh)+\frac{4}{\mu}\right)
\ee^{-(\pi/2) \mu q(4dn/\mu)},
\end{align*}
where $q(x)=x/\arsinh x$.
\end{lemma}
\begin{proof}
From~\eqref{eq:f-bound-log-algebraic-2}, it holds that
\begin{align*}
 |F(x)|\leq
 K \frac{(\pi/2)|\sinh x|(\pi/2)\cosh x}{(1+\ee^{-\pi\sinh x})^{\alpha/2}(1+\ee^{\pi\sinh x})^{\beta/2}}.
\end{align*}
Using this inequality, we have
\begin{align*}
&\left|
h\sum_{k=-\infty}^{-M-1}F(kh)
+h\sum_{k=N+1}^{\infty}F(kh)
\right|\\
&\leq h\sum_{k=-\infty}^{-M-1}
\frac{K(\pi/2)|\sinh(kh)|(\pi/2)\cosh(kh)}{(1+\ee^{-\pi\sinh(kh)})^{\alpha/2}(1+\ee^{\pi\sinh(kh)})^{\beta/2}}\\
&\quad+ h\sum_{k=N+1}^{\infty}
\frac{K(\pi/2)|\sinh(kh)|(\pi/2)\cosh(kh)}{(1+\ee^{-\pi\sinh(kh)})^{\alpha/2}(1+\ee^{\pi\sinh(kh)})^{\beta/2}}\\
& =-h\sum_{k=-\infty}^{-M-1}
\frac{K(\pi/2)^2\sinh(kh)\cosh(kh)}{(1+\ee^{\pi\sinh(kh)})^{(\alpha+\beta)/2}}
\ee^{(\pi/2)\alpha\sinh(kh)}\\
&\quad+ h\sum_{k=N+1}^{\infty}
\frac{K(\pi/2)^2\sinh(kh)\cosh(kh)}{(1+\ee^{-\pi\sinh(kh)})^{(\alpha+\beta)/2}}
\ee^{-(\pi/2)\beta\sinh(kh)}\\
&=-h\sum_{k=-\infty}^{-M-1}
\frac{K(\pi/2)^2\sinh(kh)\cosh(kh)}{(1+0)^{(\alpha+\beta)/2}}
\ee^{(\pi/2)\alpha\sinh(kh)}\\
&\quad+ h\sum_{k=N+1}^{\infty}
\frac{K(\pi/2)^2\sinh(kh)\cosh(kh)}{(1+0)^{(\alpha+\beta)/2}}
\ee^{-(\pi/2)\beta\sinh(kh)}.
\end{align*}
Using Proposition~\ref{prop:DE1-monotone},
for $Mh\geq\arsinh(4/(\pi\alpha))$ and $Nh\geq\arsinh(4/(\pi\beta))$,
we further bound the sums as
\begin{align*}
&-h\sum_{k=-\infty}^{-M-1}
K\left(\frac{\pi}{2}\right)^2\sinh(kh)\cosh(kh)
\ee^{(\pi/2)\alpha\sinh(kh)}\\
&\quad+ h\sum_{k=N+1}^{\infty}
K\left(\frac{\pi}{2}\right)^2\sinh(kh)\cosh(kh)
\ee^{-(\pi/2)\beta\sinh(kh)}\\
&\leq -\int_{-\infty}^{-Mh}
K\left(\frac{\pi}{2}\right)^2\sinh(x)\cosh(x)
\ee^{(\pi/2)\alpha\sinh x}\\
&\quad+ \int_{Nh}^{\infty}
K\left(\frac{\pi}{2}\right)^2\sinh(x)\cosh(x)
\ee^{-(\pi/2)\beta\sinh x}\\
&=\frac{K}{2\alpha}\left(\pi\sinh(Mh) + \frac{2}{\alpha}\right)
\ee^{-(\pi/2)\alpha\sinh(Mh)}
+\frac{K}{2\beta}\left(\pi\sinh(Nh) + \frac{2}{\beta}\right)
\ee^{-(\pi/2)\beta\sinh(Nh)}\\
&\leq\frac{K}{2\mu}\left(\pi\sinh(Mh) + \frac{2}{\mu}\right)
\ee^{-(\pi/2)\alpha\sinh(Mh)}
+\frac{K}{2\mu}\left(\pi\sinh(Nh) + \frac{2}{\mu}\right)
\ee^{-(\pi/2)\beta\sinh(Nh)},
\end{align*}
where $\mu=\min\{\alpha,\beta\}$ is used
at the last inequality.
Finally, using~\eqref{eq:def-MN-DE2},
we have
$\ee^{-(\pi/2)\alpha\sinh(M h)}\leq \ee^{-(\pi/2)\mu q(4 d n/\mu)}$
(because $\alpha\sinh(Mh)\geq \mu q(4 d n/\mu)$)
and $\ee^{-(\pi/2)\beta\sinh(N h)}\leq \ee^{-(\pi/2)\mu q(4 d n/\mu)}$
(because $\beta\sinh(Nh)\geq \mu q(4 d n/\mu)$),
from which we obtain the conclusion.
\end{proof}

We are now in a position to prove Theorem~\ref{thm:new-DE2}.
Note that if $h\leq \pi d$,
then with $h$ selected by~\eqref{eq:improve-h-DE2},
it holds that
\[
 Mh
\geq \arsinh\left(\frac{\mu}{\alpha}q\left(\frac{4 d n}{\mu}\right)\right)
=\arsinh\left(\frac{4d}{\alpha h}\right)\geq
\arsinh\left(\frac{4}{\pi\alpha}\right).
\]
Similarly, it holds that
\[
 Nh\geq \arsinh\left(\frac{\mu}{\beta}q\left(\frac{4 d n}{\mu}\right)\right)
=\arsinh\left(\frac{4d}{\beta h}\right)\geq
\arsinh\left(\frac{4}{\pi\beta}\right).
\]
Furthermore, if $n\geq \mu\sinh(1)/(4d)$,
i.e., if $\arsinh(4dn/\mu)\geq 1$,
then with $h$ selected by~\eqref{eq:improve-h-DE2},
it holds that
\begin{align*}
M &= \left\lceil
\frac{1}{h}\arsinh\left(\frac{4 d n/\alpha}{\arsinh(4 d n/\mu)}\right)
\right\rceil\\
&=\left\lceil
\frac{n}{\arsinh(4 d n/\mu)}\arsinh\left(\frac{4 d n/\alpha}{\arsinh(4 d n/\mu)}\right)
\right\rceil\\
&\leq\left\lceil
\frac{n}{\arsinh(4 d n/\mu)}\arsinh\left(\frac{4 d n/\alpha}{1}\right)
\right\rceil\\
&\leq\left\lceil
\frac{n}{\arsinh(4 d n/\mu)}\arsinh\left(4 d n/\mu\right)
\right\rceil\\
&=n,
\end{align*}
and similarly it holds that
\begin{align*}
N &= \left\lceil
\frac{1}{h}\arsinh\left(\frac{4 d n/\beta}{\arsinh(4 d n/\mu)}\right)
\right\rceil\\
&=\left\lceil
\frac{n}{\arsinh(4 d n/\mu)}\arsinh\left(\frac{4 d n/\beta}{\arsinh(4 d n/\mu)}\right)
\right\rceil\\
&\leq\left\lceil
\frac{n}{\arsinh(4 d n/\mu)}\arsinh\left(\frac{4 d n/\beta}{1}\right)
\right\rceil\\
&\leq\left\lceil
\frac{n}{\arsinh(4 d n/\mu)}\arsinh\left(4 d n/\mu\right)
\right\rceil\\
&=n.
\end{align*}
From the inequalities, the truncation error
in Lemma~\ref{lem:DE2-truncation} is further bounded as
\begin{align*}
&\frac{K}{2\mu}\left(\pi\sinh(Mh)+\pi\sinh(Nh)+\frac{4}{\mu}\right)
\ee^{-(\pi/2) \mu q(4dn/\mu)}\\
&\leq
\frac{K}{2\mu}\left(\pi\sinh(nh)+\pi\sinh(nh)+\frac{4}{\mu}\right)
\ee^{-(\pi/2) \mu q(4dn/\mu)}\\
&=\frac{K}{\mu}\left(\pi\sinh(nh)+\frac{2}{\mu}\right)
\ee^{-(\pi/2) \mu q(4dn/\mu)}.
\end{align*}
Therefore,
from Lemmas~\ref{lem:DE2-discretization}
and~\ref{lem:DE2-truncation},
substituting~\eqref{eq:improve-h-DE2} into $h$,
we have
\begin{align*}
\left|
\int_{-\infty}^{\infty}F(x)\dd{x}
- h \sum_{k=-M}^N F(kh)
\right| \leq C(n) n\ee^{-(\pi/2)\mu q(4 d n/\mu)},
\end{align*}
where
\[
 C(n)
=\frac{2K}{\mu^2}
\left[
\frac{2 + \pi\mu\cos d}{n(1 - \ee^{-(\pi/2)\mu q(4dn/\mu)})\cos^{(\alpha+\beta)/2}((\pi/2)\sin d)\cos^{2} d}
+2\pi d + \frac{1}{n}
\right].
\]
Furthermore, $C(n)\leq C(1)$ holds, which completes the proof of
Theorem~\ref{thm:new-DE2}.

\subsubsection{Proof of Theorem~\ref{thm:new-SE3}}

For Theorem~\ref{thm:new-SE3},
the following bounds are required.

\begin{lemma}[Okayama and Machida~{\cite[Lemma~7]{okayama17:_error_muham}}]
\label{lem:bound-SE3-1p-log}
Let $d$ be a positive constant with $d<\pi$.
Then, we have
\begin{align*}
\sup_{\zeta\in\overline{\domD_d}}
\left|
\frac{1 + \log(1+\ee^{\zeta})}{\log(1+\ee^{\zeta})}
\cdot\frac{1}{1+\ee^{-\zeta}}
\right|
&\leq \tilde{L}_d,\\
\sup_{x\in\mathbb{R}}
\left|
\frac{1 + \log(1+\ee^{x})}{\log(1+\ee^{x})}
\cdot\frac{1}{1+\ee^{-x}}
\right|
&\leq \ee^{\pi/12},
\end{align*}
where $\tilde{L}_d$ is a constant defined by~\eqref{eq:tilde-L-d}.
\end{lemma}

\begin{lemma}[Okayama et al.~{\cite[Lemma~5.3]{OkaNomuTsuru}}]
\label{lem:bound-SE-1p-log}
Let $d$ be a positive constant with $d<\pi$.
Then, we have
\begin{align*}
\sup_{\zeta\in\overline{\domD_d}}
\left|
\frac{1}{(1+\ee^{-\zeta})\log(1+\ee^{\zeta})}
\right|
&\leq \frac{1+\tilde{c}_d}{\log(2+\tilde{c}_d)},\\
\sup_{x\in\mathbb{R}}\frac{1}{(1+\ee^{-x})\log(1+\ee^x)}
&\leq 1,
\end{align*}
where $\tilde{c}_d=1/\cos(d/2)$.
\end{lemma}

\begin{lemma}
\label{lem:bound-SE3-log}
Let $d$ be a positive constant with $d<\pi$.
For all real numbers $x$ and $y$ with $|y|\leq d$,
we have
\begin{align*}
\left|
\log(\log(1+\ee^{x+\ii y}))
\right|
&\leq \frac{1+\tilde{c}_d}{\log(2+ \tilde{c}_d)}\sqrt{x^2+y^2} - \log(\log 2),
\\
\left|
\log(\log(1+\ee^{x}))
\right|
&\leq |x| - \log(\log 2),
\end{align*}
where $\tilde{c}_d=1/\cos(d/2)$.
\end{lemma}
\begin{proof}
From
$\{\log(\log(1+\ee^{\zeta}))\}'= 1/\{(1+\ee^{-\zeta})\log(1+\ee^{\zeta})\}$,
it holds that
\begin{align*}
 \int_0^{x+\ii y}\frac{1}{(1+\ee^{-\zeta})\log(1+\ee^{\zeta})}\dd{\zeta}
&=\left[\log(\log(1+\ee^{\zeta}))
\right]_{\zeta=0}^{\zeta=x+\ii y}\\
&=\log(\log(1+\ee^{x+\ii y})) - \log(\log 2).
\end{align*}
Using this equality and Lemma~\ref{lem:bound-SE-1p-log},
we have
\begin{align*}
 |\log(\log(1+\ee^{x+\ii y}))|
&=
\left|
\log(\log 2)+
\int_0^{x+\ii y}\frac{1}{(1+\ee^{-\zeta})\log(1+\ee^{\zeta})}\dd{\zeta}
\right|\\
&\leq |\log(\log 2)|
+
\int_0^{x+\ii y}
\left|\frac{1}{(1+\ee^{-\zeta})\log(1+\ee^{\zeta})}\right|
|\dd{\zeta}|\\
&\leq |\log(\log 2)|
+\int_0^{x+\ii y}\frac{1+\tilde{c}_d}{\log(2+\tilde{c}_d)}|\dd\zeta|\\
&= -\log(\log 2)+\frac{1+\tilde{c}_d}{\log(2+\tilde{c}_d)}\sqrt{x^2+y^2},
\end{align*}
which is the first inequality.
In the same manner, we have
\begin{align*}
 |\log(\log(1+\ee^{x}))|
&=\left|
\log(\log 2)+
\int_0^{x}\frac{1}{(1+\ee^{-t})\log(1+\ee^{t})}\dd{t}
\right|\\
&\leq -\log(\log 2)
+\int_0^{x}\frac{1}{(1+\ee^{-t})\log(1+\ee^{t})}|\dd{t}|\\
&\leq -\log(\log 2) + \int_0^x 1|\dd{t}|\\
&= -\log(\log 2) + |x|,
\end{align*}
which is the second inequality.
\end{proof}

Using these bounds,
we bound the discretization error as follows.

\begin{lemma}
\label{lem:SE3-discretization}
Let $K$, $\alpha$, $\beta$ and $d$ be positive constants with
$\alpha\leq 1$ and $d<\pi$.
Assume that $f$ is analytic on $\psi_3(\domD_d)$, and
satisfies~\eqref{eq:f-bound-log-algebraic-3} for all $z\in\psi_3(\domD_d)$.
Let $\mu=\min\{\alpha,\beta\}$. Then, putting
$F(x)=f(\psi_3(x))\psi_3'(x)$, we have~\eqref{eq:discretization-error},
where
\[
\mathcal{N}(F,d)
\leq \frac{4K\tilde{L}_d^{1-\alpha}}{\mu\cos^{\alpha+\beta}(d/2)}
\left\{\left(d + \frac{1}{\mu}\right)\frac{1+\tilde{c}_d}{\log(2+\tilde{c}_d)}
-\log(\log 2)
\right\},
\]
where $\tilde{c}_d=1/\cos(d/2)$
and $\tilde{L}_d$ is a constant defined by~\eqref{eq:tilde-L-d}.
\end{lemma}
\begin{proof}
It suffices to show that $F\in\mathbf{B}(\domD_d)$.
Because $f(\psi_3(\cdot))$ is analytic on $\domD_d$ and
$\psi_3'$ is analytic on $\domD_{\pi}$,
$F$ is analytic on $\domD_d$.
Next, we show~\eqref{eq:F-imag-int-0}.
From~\eqref{eq:f-bound-log-algebraic-3}
and Lemmas~\ref{lem:bound-SE3-1p-log},~\ref{lem:bound-SE-1p-exp}
and~\ref{lem:bound-SE3-log}, it holds
for $\zeta=x+\ii y\in \domD_d$ that
\begin{align}
 |F(\zeta)|&\leq
K\left|\frac{1+\log(1+\ee^{\zeta})}{\log(1+\ee^{\zeta})}\right|^{1-\alpha}
\frac{1}{|1+\ee^{\zeta}|^{\beta}}
\left|\log(\log(1+\ee^{\zeta}))\right|
\frac{1}{|1+\ee^{-\zeta}|}
\nonumber\\
&\leq K\left\{\tilde{L}_d|1+\ee^{-\zeta}|\right\}^{1-\alpha}
\frac{1}{|1+\ee^{\zeta}|^{\beta}}
\left|\log(\log(1+\ee^{\zeta}))\right|
\frac{1}{|1+\ee^{-\zeta}|}\nonumber\\
&=\frac{K\tilde{L}_d^{1-\alpha}}{|1+\ee^{-\zeta}|^{\alpha}|1+\ee^{\zeta}|^{\beta}}
\left|\log(\log(1+\ee^{\zeta}))\right|\nonumber\\
&\leq
\frac{K\tilde{L}_d^{1-\alpha}}{(1+\ee^{-x})^{\alpha}(1+\ee^{x})^{\beta}\cos^{\alpha+\beta}(y/2)}
\left(\tilde{\lambda}_d
\sqrt{x^2+y^2} - \log(\log 2)\right),
\label{eq:bound-F-SE3}
\end{align}
where we put $\tilde{\lambda}_d=(1+\tilde{c}_d)/\log(2+\tilde{c}_d)$.
Using this inequality, we have
\begin{align*}
&\int_{-d}^d\left|F(x + \ii y)\right|\dd y\\
&\leq
\frac{K\tilde{L}_d^{1-\alpha}\sqrt{x^2+1}}{(1+\ee^{-x})^{\alpha}(1+\ee^{x})^{\beta}}
\int_{-d}^d\frac{1}{\cos^{\alpha+\beta}(y/2)}
\left(\tilde{\lambda}_d\sqrt{\frac{x^2}{x^2+1}+\frac{y^2}{x^2+1}} + \frac{-\log(\log 2)}{\sqrt{x^2+1}}\right)\dd{y}\\
& <
\frac{K\tilde{L}_d^{1-\alpha}\sqrt{x^2+1}}{(1+\ee^{-x})^{\alpha}(1+\ee^{x})^{\beta}}
\int_{-d}^d\frac{1}{\cos^{\alpha+\beta}(y/2)}
\left(\tilde{\lambda}_d\sqrt{1+\frac{y^2}{0^2 + 1}} + \frac{-\log(\log 2)}{\sqrt{0^2 + 1}}\right)\dd{y}\\
&\to 0 \cdot
\int_{-d}^d\frac{1}{\cos^{\alpha+\beta}(y/2)}
\left(\tilde{\lambda}_d\sqrt{1+y^2} -\log(\log 2)\right)\dd{y}
\quad (x \to \pm\infty),
\end{align*}
which shows~\eqref{eq:F-imag-int-0}.
Finally, we estimate $\mathcal{N}(F,d)$.
Using~\eqref{eq:bound-F-SE3},
we have
\begin{align*}
&\int_{-\infty}^{\infty} \left\{|F(x + \ii y)|+|F(x - \ii y)|\right\}\dd x\\
&\leq\frac{K\tilde{L}_d^{1-\alpha}}{\cos^{\alpha+\beta}(y/2)}
\int_{-\infty}^{\infty}
\frac{\tilde{\lambda}_d\sqrt{x^2+y^2} - \log(\log 2)}{(1+\ee^{-x})^{\alpha}(1+\ee^{x})^{\beta}}
\dd{x}\\
&\quad +\frac{K\tilde{L}_d^{1-\alpha}}{\cos^{\alpha+\beta}(-y/2)}
\int_{-\infty}^{\infty}
\frac{\tilde{\lambda}_d\sqrt{x^2+(-y)^2} - \log(\log 2)}{(1+\ee^{-x})^{\alpha}(1+\ee^{x})^{\beta}}
\dd{x}\\
&\leq\frac{2K\tilde{L}_d^{1-\alpha}}{\cos^{\alpha+\beta}(d/2)}
\int_{-\infty}^{\infty}
\frac{\tilde{\lambda}_d\sqrt{x^2+d^2} - \log(\log 2)}{(1+\ee^{-x})^{\alpha}(1+\ee^{x})^{\beta}}
\dd{x},
\end{align*}
which holds for $y\in [-d,d]$ (note that $d<\pi$).
Therefore, the inequality remains valid
when taking the limit $y\to d - 0$.
Using $\mu=\min\{\alpha,\beta\}$,
we bound the integral as
\begin{align*}
\int_{-\infty}^{\infty}
\frac{\tilde{\lambda}_d\sqrt{x^2+d^2} - \log(\log 2)}{(1+\ee^{-x})^{\alpha}(1+\ee^{x})^{\beta}}
\dd{x}
&\leq \int_{-\infty}^{\infty}
\frac{\tilde{\lambda}_d\sqrt{x^2+d^2} - \log(\log 2)}{(1+\ee^{-x})^{\mu}
(1+\ee^{x})^{\mu}}
\dd{x}\\
&=2 \int_0^{\infty}
\frac{\tilde{\lambda}_d\sqrt{x^2+d^2} - \log(\log 2)}{(1+\ee^{-x})^{2\mu}}
\ee^{-\mu x}
\dd{x}\\
&\leq 2 \int_0^{\infty}
\frac{\tilde{\lambda}_d\sqrt{x^2+d^2} - \log(\log 2)}{(1+0)^{2\mu}}
\ee^{-\mu x}\dd x.
\end{align*}
Furthermore, using integration by parts, we have
\begin{align*}
&2 \int_0^{\infty}
\left(\tilde{\lambda}_d\sqrt{x^2+d^2} - \log(\log 2)\right)
\ee^{-\mu x}\dd x\\
&=2 \int_0^{\infty}
\left(\tilde{\lambda}_d\sqrt{x^2+d^2} - \log(\log 2)\right)
\left(\frac{\ee^{-\mu x}}{-\mu}\right)'\dd x\\
&=2\left[
\left(\tilde{\lambda}_d\sqrt{x^2+d^2} - \log(\log 2)\right)
\cdot\frac{\ee^{-\mu x}}{-\mu}
\right]_{x=0}^{x=\infty}
+\frac{2\tilde{\lambda}_d}{\mu}\int_0^{\infty}
\frac{x}{\sqrt{x^2+d^2}}\ee^{-\mu x}\dd{x}\\
&\leq\frac{2\left(\tilde{\lambda}_d d - \log(\log 2)\right)}{\mu}
+\frac{2\tilde{\lambda}_d}{\mu}\int_0^{\infty}
\frac{x}{\sqrt{x^2 + 0}}\ee^{-\mu x}\dd{x}\\
&=\frac{2\left(\tilde{\lambda}_d d - \log(\log 2)\right)}{\mu}
+\frac{2\tilde{\lambda}_d}{\mu^2}.
\end{align*}
Thus, we obtain the conclusion.
\end{proof}

Next, we bound the truncation error as follows.

\begin{lemma}
\label{lem:SE3-truncation}
Let $K$, $\alpha$ and $\beta$ be positive constants with $\alpha\leq 1$.
Assume that $f$ satisfies~\eqref{eq:f-bound-log-algebraic-3}
for all $z\in (0, \infty)$.
Let $\mu=\min\{\alpha,\beta\}$,
let $n$ be positive integer,
and
let $M$ and $N$ be selected by~\eqref{eq:def-MN-SE}.
Let $Mh\geq 1/\alpha$ and $Nh\geq 1/\beta$ be satisfied.
Then, putting $F(x)=f(\psi_3(x))\psi_3'(x)$, we have
\begin{align*}
\left|
h\sum_{k=-\infty}^{-M-1}F(kh)
+h\sum_{k=N+1}^{\infty}F(kh)
\right|
\leq \frac{2K\ee^{\pi(1-\alpha)/12}}{\mu}
\left(nh - \log(\log 2) + \frac{1}{\mu}\right)
\ee^{-\mu n h}.
\end{align*}
\end{lemma}
\begin{proof}
From~\eqref{eq:f-bound-log-algebraic-3}
and Lemmas~\ref{lem:bound-SE3-1p-log} and~\ref{lem:bound-SE3-log},
it holds that
\begin{align*}
 |F(x)|&\leq K
\left|\frac{\log(1+\ee^{x})}{1+\log(1+\ee^x)}\right|^{\alpha - 1}
\left|\ee^{-\log(1+\ee^{x})}\right|^{\beta}
|\log(\log(1+\ee^x))|
\left|\frac{1}{1+\ee^{-x}}\right|\\
&=K\left\{\frac{1+\log(1+\ee^x)}{\log(1+\ee^x)}\right\}^{1-\alpha}
\frac{1}{(1+\ee^x)^{\beta}}|\log(\log(1+\ee^x))|
\frac{1}{1+\ee^{-x}}\\
&\leq K\left\{\ee^{\pi/12}(1+\ee^{-x})\right\}^{1-\alpha}
\frac{1}{(1+\ee^x)^{\beta}}
\left\{|x| - \log(\log 2)\right\}
\frac{1}{1+\ee^{-x}}\\
&=\frac{K\ee^{\pi(1-\alpha)/12}\left\{|x|-\log(\log 2)\right\}}
{(1+\ee^{-x})^{\alpha}(1+\ee^x)^{\beta}}.
\end{align*}
Using this inequality, we have
\begin{align*}
&\left|
h\sum_{k=-\infty}^{-M-1}F(kh)
+h\sum_{k=N+1}^{\infty}F(kh)
\right|\\
&\leq K\ee^{\pi(1-\alpha)/12}\left[
h\sum_{k=-\infty}^{-M-1}
\frac{|kh|-\log(\log 2)}
{(1+\ee^{-kh})^{\alpha}(1+\ee^{kh})^{\beta}}
+ h\sum_{k=N+1}^{\infty}
\frac{|kh|-\log(\log 2)}
{(1+\ee^{-kh})^{\alpha}(1+\ee^{kh})^{\beta}}
\right]\\
&=K\ee^{\pi(1-\alpha)/12}\left[
 h\sum_{k=-\infty}^{-M-1}
\frac{-kh-\log(\log 2)}
{(1+\ee^{kh})^{\alpha+\beta}}\ee^{\alpha kh}
+ h\sum_{k=N+1}^{\infty}
\frac{kh-\log(\log 2)}
{(1+\ee^{-kh})^{\alpha+\beta}}\ee^{-\beta k h}
\right]\\
&\leq K\ee^{\pi(1-\alpha)/12}\left[
 h\sum_{k=-\infty}^{-M-1}
\frac{-kh-\log(\log 2)}
{(1+0)^{\alpha+\beta}}\ee^{\alpha kh}
+ h\sum_{k=N+1}^{\infty}
\frac{kh-\log(\log 2)}
{(1+0)^{\alpha+\beta}}\ee^{-\beta k h}
\right].
\end{align*}
Using Proposition~\ref{prop:SE1-monotone},
for $Mh\geq 1/\alpha$ and $Nh\geq 1/\beta$,
we further bound the sums as
\begin{align*}
&K\ee^{\pi(1-\alpha)/12}\left[
h\sum_{k=-\infty}^{-M-1}
\left\{-kh-\log(\log 2)\right\}\ee^{\alpha k h}
+ h\sum_{k=N+1}^{\infty}
\left\{kh-\log(\log 2)\right\}
\ee^{-\beta k h}\right]\\
&\leq K\ee^{\pi(1-\alpha)/12}\left[
\int_{-\infty}^{-Mh}
\left\{-x-\log(\log 2)\right\}\ee^{\alpha x}\dd{x}
+ \int_{Nh}^{\infty}
\left\{x-\log(\log 2)\right\}
\ee^{-\beta x}\dd{x}\right]\\
&=K\ee^{\pi(1-\alpha)/12}\left[
\left(Mh -\log(\log 2) + \frac{1}{\alpha}\right)\frac{\ee^{-\alpha Mh}}{\alpha}
+\left(Nh -\log(\log 2) + \frac{1}{\beta}\right)\frac{\ee^{-\beta Nh}}{\beta}\right]\\
&\leq K\ee^{\pi(1-\alpha)/12}\left[
\left(Mh -\log(\log 2) + \frac{1}{\mu}\right)\frac{\ee^{-\alpha Mh}}{\mu}
+\left(Nh -\log(\log 2) + \frac{1}{\mu}\right)\frac{\ee^{-\beta Nh}}{\mu}\right],
\end{align*}
where $\mu=\min\{\alpha,\beta\}$ is used
at the last inequality.
Finally, using~\eqref{eq:def-MN-SE},
we have $Mh\leq nh$ (because $M\leq n$),
$Nh\leq nh$ (because $N\leq n$),
$\ee^{-\alpha M h}\leq \ee^{-\mu n h}$
(because $\alpha M\geq \mu n$)
and $\ee^{-\beta N h}\leq \ee^{-\mu n h}$ (because $\beta N\geq \mu n$),
from which we obtain the conclusion.
\end{proof}

We are now in a position to prove Theorem~\ref{thm:new-SE3}.
Note that if $n\geq 1/(2\pi d \mu)$,
then with $h$ selected by~\eqref{eq:standard-h-SE},
the inequalities~\eqref{eq:Mh-bounded-by-1-alpha}
and~\eqref{eq:Nh-bounded-by-1-beta} hold.
Therefore,
from Lemmas~\ref{lem:SE3-discretization}
and~\ref{lem:SE3-truncation},
substituting~\eqref{eq:standard-h-SE} into $h$,
we have
\begin{align*}
\left|
\int_{-\infty}^{\infty}F(x)\dd{x}
- h \sum_{k=-M}^N F(kh)
\right| \leq C(n) \sqrt{n}\ee^{-\sqrt{2\pi d \mu n}},
\end{align*}
where
\begin{align*}
 C(n)
&=\frac{2 K}{\mu^2}
\left[
\left(\frac{2\tilde{L}_d^{1-\alpha}(\mu d + 1)}{\cos^{\alpha+\beta}(d/2)}
\cdot\frac{1+\tilde{c}_d}{\log(2+\tilde{c}_d)} - \mu\log(\log 2)
\right)\frac{1}{\sqrt{n}(1 - \ee^{-\sqrt{2\pi d \mu n}})}\right.\\
&\left.\quad\quad\quad
+\ee^{\pi(1-\alpha)/12}\left(\sqrt{2\pi d \mu}
+\frac{1 - \mu\log(\log 2)}{\sqrt{n}}
\right)
\right].
\end{align*}
Furthermore, $C(n)\leq C(1)$ holds, which completes the proof of
Theorem~\ref{thm:new-SE3}.

\subsubsection{Proof of Theorem~\ref{thm:new-DE3}}

For Theorem~\ref{thm:new-DE3},
the following bounds are required.

\begin{lemma}[Okayama~{\cite[Lemma~7]{okayama13:_error_sinc_sinc}}]
\label{lem:bound-DE3-1p-log}
Let $d$ be a positive constant with $d<\pi/2$.
Then, we have
\begin{align*}
\sup_{\zeta\in\overline{\domD_d}}
\left|
\frac{1 + \log(1+\ee^{\pi\sinh\zeta})}{\log(1+\ee^{\pi\sinh\zeta})}
\cdot\frac{1}{1+\ee^{-\pi\sinh\zeta}}
\right|
&\leq L_d,\\
\sup_{x\in\mathbb{R}}
\left|
\frac{1 + \log(1+\ee^{\pi\sinh x})}{\log(1+\ee^{\pi\sinh x})}
\cdot\frac{1}{1+\ee^{-\pi\sinh x}}
\right|
&\leq \ee^{\pi/12},
\end{align*}
where $L_d$ is a constant defined by~\eqref{eq:L-d}.
\end{lemma}

\begin{lemma}[Okayama et al.~{\cite[Lemma~6.12]{OkaHaraGoto}}]
\label{lem:bound-DE-1p-log}
Let $d$ be a positive constant with $d<\pi/2$.
Then, it holds that
\begin{align*}
\sup_{\zeta\in\overline{\domD_d}}
\left|
\frac{1}{(1+\ee^{-\pi\sinh\zeta})\log(1+\ee^{\pi\sinh\zeta})}
\right|
&\leq \frac{1+c_d}{\log(2+c_d)},
\end{align*}
where $c_d=1/\cos((\pi/2)\sin d)$.
\end{lemma}

\begin{lemma}
\label{lem:bound-DE3-log}
Let $d$ be a positive constant with $d<\pi/2$.
For all real numbers $x$ and $y$ with $|y|\leq d$,
we have
\begin{align*}
\left|
\log(\log(1+\ee^{\pi\sinh(x+\ii y)}))
\right|
&\leq \frac{\pi(1+c_d)}{\log(2+ c_d)}\left\{1+|y|\right\}\cosh x - \log(\log 2),
\\
\left|
\log(\log(1+\ee^{\pi\sinh x}))
\right|
&\leq |\pi\sinh x| - \log(\log 2),
\end{align*}
where $c_d=1/\cos((\pi/2)\sin d)$.
\end{lemma}
\begin{proof}
From
\[
 \left\{\log(\log(1+\ee^{\pi\sinh\zeta}))\right\}'
= \frac{\pi\cosh\zeta}{(1+\ee^{-\pi\sinh\zeta})\log(1+\ee^{\pi\sinh\zeta})},
\]
it holds that
\begin{align*}
 \int_0^{x+\ii y}
\frac{\pi\cosh\zeta}{(1+\ee^{-\pi\sinh\zeta})\log(1+\ee^{\pi\sinh\zeta})}
\dd{\zeta}
&=\left[\log(\log(1+\ee^{\pi\sinh\zeta}))
\right]_{\zeta=0}^{\zeta=x+\ii y}\\
&=\log(\log(1+\ee^{\pi\sinh(x+\ii y)})) - \log(\log 2).
\end{align*}
Using this equality and Lemma~\ref{lem:bound-DE-1p-log},
we have
\begin{align*}
& |\log(\log(1+\ee^{\pi\sinh(x+\ii y)}))|\\
&=
\left|
\log(\log 2)+
\int_0^{x+\ii y}
\frac{\pi\cosh\zeta}{(1+\ee^{-\pi\sinh\zeta})\log(1+\ee^{\pi\sinh\zeta})}
\dd{\zeta}
\right|\\
&\leq
\left|
\log(\log 2)
\right|
+\left|
\int_0^{x}
\frac{\pi\cosh\zeta}{(1+\ee^{-\pi\sinh\zeta})\log(1+\ee^{\pi\sinh\zeta})}
\dd{\zeta}\right.\\
&\quad\quad\quad\quad\quad\quad+\left.
\int_x^{x+\ii y}
\frac{\pi\cosh\zeta}{(1+\ee^{-\pi\sinh\zeta})\log(1+\ee^{\pi\sinh\zeta})}
\dd{\zeta}\right|\\
&\leq
-\log(\log 2)
+\frac{\pi(1+c_d)}{\log(2+c_d)}\int_0^{x}|\cosh\zeta||\dd\zeta|
+\frac{\pi(1+c_d)}{\log(2+c_d)}\int_x^{x+\ii y}|\cosh\zeta||\dd\zeta|\\
&= -\log(\log 2)
+\frac{\pi(1+c_d)}{\log(2+c_d)}
\left(\int_0^{x}\cosh t|\dd{t}|+\int_0^y|\cosh(x+\ii u)||\dd{u}|\right)\\
&\leq -\log(\log 2)
+\frac{\pi(1+c_d)}{\log(2+c_d)}
\left(|\sinh x|+|y|\cosh x\right)\\
&\leq -\log(\log 2)
+\frac{\pi(1+c_d)}{\log(2+c_d)}
\left(1+|y|\right)\cosh x,
\end{align*}
which is the first inequality.
The second inequality is obtained by replacing
$x$ with $\pi\sinh x$ in the second inequality
of Lemma~\ref{lem:bound-SE3-log}.
\end{proof}

Using these bounds,
we bound the discretization error as follows.

\begin{lemma}
\label{lem:DE3-discretization}
Let $K$, $\alpha$, $\beta$ and $d$ be positive constants with
$\alpha\leq 1$ and $d<\pi/2$.
Assume that $f$ is analytic on $\phi_3(\domD_d)$, and
satisfies~\eqref{eq:f-bound-log-algebraic-3} for all $z\in\phi_3(\domD_d)$.
Let $\mu=\min\{\alpha,\beta\}$. Then, putting
$F(x)=f(\phi_3(x))\phi_3'(x)$, we have~\eqref{eq:discretization-error},
where
\[
\mathcal{N}(F,d)
\leq \frac{4KL_d^{1-\alpha}}{\mu\cos^{\alpha+\beta}((\pi/2)\sin d)\cos d}
\left\{\left(\pi + \frac{1}{\mu\cos d}\right)\frac{(1+c_d)(1+d)}{\log(2+c_d)}
-\log(\log 2)
\right\},
\]
where $c_d=1/\cos((\pi/2)\sin d)$
and $L_d$ is a constant defined by~\eqref{eq:L-d}.
\end{lemma}
\begin{proof}
It suffices to show that $F\in\mathbf{B}(\domD_d)$.
Because $f(\phi_3(\cdot))$ is analytic on $\domD_d$ and
$\phi_3'$ is analytic on $\domD_{\pi/2}$,
$F$ is analytic on $\domD_d$.
Next, we show~\eqref{eq:F-imag-int-0}.
From~\eqref{eq:f-bound-log-algebraic-3}
and Lemmas~\ref{lem:bound-DE3-1p-log},~\ref{lem:bound-DE-1p-exp}
and~\ref{lem:bound-DE3-log}, it holds
for $\zeta=x+\ii y\in \domD_d$ that
\begin{align}
& |F(\zeta)|\nonumber\\
&\leq
K\left|\frac{1+\log(1+\ee^{\pi\sinh\zeta})}{\log(1+\ee^{\pi\sinh\zeta})}\right|^{1-\alpha}
\frac{1}{|1+\ee^{\pi\sinh\zeta}|^{\beta}}
\left|\log(\log(1+\ee^{\pi\sinh\zeta}))\right|
\frac{\pi|\cosh\zeta|}{|1+\ee^{-\pi\sinh\zeta}|}
\nonumber\\
&\leq K\left\{L_d|1+\ee^{-\pi\sinh\zeta}|\right\}^{1-\alpha}
\frac{1}{|1+\ee^{\pi\sinh\zeta}|^{\beta}}
\left|\log(\log(1+\ee^{\pi\sinh\zeta}))\right|
\frac{\pi|\cosh\zeta|}{|1+\ee^{-\pi\sinh\zeta}|}\nonumber\\
&=\frac{KL_d^{1-\alpha}\left|\log(\log(1+\ee^{\pi\sinh\zeta}))\right|\pi|\cosh\zeta|}{|1+\ee^{-\pi\sinh\zeta}|^{\alpha}|1+\ee^{\pi\sinh\zeta}|^{\beta}}
\nonumber\\
&\leq
\frac{KL_d^{1-\alpha}\left\{\pi\lambda_d(1+|y|)\cosh x
 - \log(\log 2)\right\}\pi\cosh x}{(1+\ee^{-\pi\sinh(x)\cos y})^{\alpha}(1+\ee^{\pi\sinh(x)\cos y})^{\beta}\cos^{\alpha+\beta}((\pi/2)\sin y)}\nonumber\\
&\leq
\frac{KL_d^{1-\alpha}\left\{\pi\lambda_d(1+|y|)\cosh x
 - \log(\log 2)\right\}\pi\cosh x}{(1+\ee^{-\pi\sinh(x)\cos y})^{\mu}(1+\ee^{\pi\sinh(x)\cos y})^{\mu}\cos^{\alpha+\beta}((\pi/2)\sin y)},
\label{eq:bound-F-DE3}
\end{align}
where we put $\lambda_d=(1+c_d)/\log(2+c_d)$.
Using this inequality~\eqref{eq:bound-F-DE3}, we have
\begin{align*}
&\int_{-d}^d\left|F(x + \ii y)\right|\dd y\\
&\leq
\int_{-d}^d
\frac{KL_d^{1-\alpha}\left\{\pi\lambda_d(1+|y|)\cosh x
 - \log(\log 2)\right\}\pi\cosh x}{(1+\ee^{-\pi\sinh(x)\cos y})^{\mu}(1+\ee^{\pi\sinh(x)\cos y})^{\mu}\cos^{\alpha+\beta}((\pi/2)\sin y)}\dd{y}\\
&\leq
\frac{KL_d^{1-\alpha}\left\{\pi\lambda_d(1+ d)\cosh x
 - \log(\log 2)\right\}\pi\cosh x}{(1+\ee^{-\pi\sinh(x)\cos d})^{\mu}(1+\ee^{\pi\sinh(x)\cos d})^{\mu}\cos^{\alpha+\beta}((\pi/2)\sin d)}
\cdot
\int_{-d}^d \dd{y}\\
&\to 0 \cdot 2d
\quad (x \to \pm\infty),
\end{align*}
which shows~\eqref{eq:F-imag-int-0}.
Finally, we estimate $\mathcal{N}(F,d)$.
Using~\eqref{eq:bound-F-DE3},
we have
\begin{align*}
&\int_{-\infty}^{\infty} \left\{|F(x + \ii y)|+|F(x - \ii y)|\right\}\dd x\\
&\leq\frac{KL_d^{1-\alpha}}{\cos^{\alpha+\beta}((\pi/2)\sin y)}
\int_{-\infty}^{\infty}
\frac{\left\{\pi\lambda_d(1+|y|)\cosh x - \log(\log 2)\right\}\pi\cosh x}{(1+\ee^{-\pi\sinh(x)\cos y})^{\mu}(1+\ee^{\pi\sinh(x)\cos y})^{\mu}}
\dd{x}\\
&\quad+\frac{KL_d^{1-\alpha}}{\cos^{\alpha+\beta}((\pi/2)\sin(-y))}
\int_{-\infty}^{\infty}
\frac{\left\{\pi\lambda_d(1+|-y|)\cosh x - \log(\log 2)\right\}\pi\cosh x}{(1+\ee^{-\pi\sinh(x)\cos(-y)})^{\mu}(1+\ee^{\pi\sinh(x)\cos(-y)})^{\mu}}
\dd{x}\\
&\leq
\frac{2KL_d^{1-\alpha}}{\cos^{\alpha+\beta}((\pi/2)\sin d)}
\int_{-\infty}^{\infty}
\frac{\left\{\pi\lambda_d(1+d)\cosh x - \log(\log 2)\right\}\pi\cosh x}{(1+\ee^{-\pi\sinh(x)\cos d})^{\mu}(1+\ee^{\pi\sinh(x)\cos d})^{\mu}}
\dd{x},
\end{align*}
which holds for $y\in [-d,d]$ (note that $d<\pi/2$).
Therefore, the inequality remains valid
when taking the limit $y\to d - 0$.
We bound the integral as
\begin{align*}
&\int_{-\infty}^{\infty}
\frac{\left\{\pi\lambda_d(1+d)\cosh x - \log(\log 2)\right\}\pi\cosh x}{(1+\ee^{-\pi\sinh(x)\cos d})^{\mu}(1+\ee^{\pi\sinh(x)\cos d})^{\mu}}
\dd{x}\\
&=2 \int_0^{\infty}
\frac{\left\{\pi\lambda_d(1+d)\cosh x - \log(\log 2)\right\}\pi\cosh x}{(1+\ee^{-\pi\sinh(x)\cos d})^{2\mu}}
\ee^{-\pi\mu\sinh(x)\cos d}
\dd{x}\\
&\leq 2 \int_0^{\infty}
\frac{\left\{\pi\lambda_d(1+d)\cosh x - \log(\log 2)\right\}\pi\cosh x}{(1+0)^{2\mu}}
\ee^{-\pi\mu\sinh(x)\cos d}\dd x.
\end{align*}
Furthermore, using integration by parts, we have
\begin{align*}
&2 \int_0^{\infty}
\left\{\pi\lambda_d(1+d)\cosh x - \log(\log 2)\right\}\pi\cosh x
\ee^{-\pi\mu\sinh(x)\cos d}\dd x\\
&=2 \int_0^{\infty}
\left\{\pi\lambda_d(1+d)\cosh x - \log(\log 2)\right\}
\left(\frac{\ee^{-\pi\mu\sinh(x)\cos d}}{-\mu\cos d}\right)'\dd x\\
&=2\left[
\left\{\pi\lambda_d(1+d)\cosh x - \log(\log 2)\right\}
\cdot\left(\frac{\ee^{-\pi\mu\sinh(x)\cos d}}{-\mu\cos d}\right)
\right]_{x=0}^{x=\infty}\\
&\quad+\frac{2\lambda_d(1+d)}{\mu\cos d}\int_0^{\infty}
\pi\sinh(x)\ee^{-\pi\mu\sinh(x)\cos d}\dd{x}\\
&\leq\frac{2\left\{\pi\lambda_d (1+d) - \log(\log 2)\right\}}{\mu\cos d}
+\frac{2\lambda_d(1+d)}{\mu\cos d}\int_0^{\infty}
\pi\cosh(x)\ee^{-\pi\mu\sinh(x)\cos d}\dd{x}\\
&=\frac{2\left\{\pi\lambda_d (1+d) - \log(\log 2)\right\}}{\mu\cos d}
+\frac{2\lambda_d(1+d)}{\mu^2\cos^2 d}.
\end{align*}
Thus, we obtain the conclusion.
\end{proof}

Next, we bound the truncation error as follows.

\begin{lemma}
\label{lem:DE3-truncation}
Let $K$, $\alpha$ and $\beta$ be positive constants with $\alpha\leq 1$.
Assume that $f$ satisfies~\eqref{eq:f-bound-log-algebraic-3}
for all $z\in (0, \infty)$.
Let $\mu=\min\{\alpha,\beta\}$,
let $n$ be positive integer,
and
let $M$ and $N$ be selected by~\eqref{eq:def-MN-DE}.
Let $Mh\geq \arsinh(2/(\pi\alpha))$ and $Nh\geq \arsinh(2/(\pi\beta))$
be satisfied.
Then, putting $F(x)=f(\phi_3(x))\phi_3'(x)$, we have
\begin{align*}
&\left|
h\sum_{k=-\infty}^{-M-1}F(kh)
+h\sum_{k=N+1}^{\infty}F(kh)
\right|\\
&\leq \frac{K\ee^{\pi(1-\alpha)/12}}{\mu}
\left(\pi\sinh(Mh) + \pi\sinh(Nh) - 2\log(\log 2) + \frac{2}{\mu}\right)
\ee^{-\pi\mu q(2 d n/\mu)},
\end{align*}
where $q(x)=x/\arsinh x$.
\end{lemma}
\begin{proof}
From~\eqref{eq:f-bound-log-algebraic-3}
and Lemmas~\ref{lem:bound-DE3-1p-log} and~\ref{lem:bound-DE3-log},
it holds that
\begin{align*}
 &|F(x)|\\
&\leq K
\left\{\frac{1+\log(1+\ee^{\pi\sinh x})}{\log(1+\ee^{\pi\sinh x})}\right\}^{1-\alpha}
\frac{1}{(1+\ee^{\pi\sinh x})^{\beta}}
|\log(\log(1+\ee^{\pi\sinh x}))|
\frac{\pi\cosh x}{1+\ee^{-\pi\sinh x}}\\
&\leq K\left\{\ee^{\pi/12}(1+\ee^{-\pi\sinh x})\right\}^{1-\alpha}
\frac{1}{(1+\ee^{\pi\sinh x})^{\beta}}
\left\{\pi|\sinh x| - \log(\log 2)\right\}
\frac{\pi\cosh x}{1+\ee^{-\pi\sinh x}}\\
&=\frac{K\ee^{\pi(1-\alpha)/12}\left\{\pi|\sinh x|-\log(\log 2)\right\}\pi\cosh x}
{(1+\ee^{-\pi\sinh x})^{\alpha}(1+\ee^{\pi\sinh x})^{\beta}}.
\end{align*}
Using this inequality, we have
\begin{align*}
&\left|
h\sum_{k=-\infty}^{-M-1}F(kh)
+h\sum_{k=N+1}^{\infty}F(kh)
\right|\\
&\leq h\sum_{k=-\infty}^{-M-1}
\frac{K\ee^{\pi(1-\alpha)/12}\left\{\pi|\sinh(kh)|-\log(\log 2)\right\}\pi\cosh(kh)}
{(1+\ee^{-\pi\sinh(kh)})^{\alpha}(1+\ee^{\pi\sinh(kh)})^{\beta}}\\
&\quad +h\sum_{k=N+1}^{\infty}
\frac{K\ee^{\pi(1-\alpha)/12}\left\{\pi|\sinh(kh)|-\log(\log 2)\right\}\pi\cosh(kh)}
{(1+\ee^{-\pi\sinh(kh)})^{\alpha}(1+\ee^{\pi\sinh(kh)})^{\beta}}\\
&=h\sum_{k=-\infty}^{-M-1}
\frac{K\ee^{\pi(1-\alpha)/12}\left\{-\pi\sinh(kh)-\log(\log 2)\right\}\pi\cosh(kh)}
{(1+\ee^{\pi\sinh(kh)})^{\alpha+\beta}}\ee^{\pi\alpha\sinh(kh)}\\
&\quad +h\sum_{k=N+1}^{\infty}
\frac{K\ee^{\pi(1-\alpha)/12}\left\{\pi\sinh(kh)-\log(\log 2)\right\}\pi\cosh(kh)}
{(1+\ee^{-\pi\sinh(kh)})^{\alpha+\beta}}\ee^{-\pi\beta\sinh(kh)}\\
&\leq h\sum_{k=-\infty}^{-M-1}
\frac{K\ee^{\pi(1-\alpha)/12}\left\{-\pi\sinh(kh)-\log(\log 2)\right\}\pi\cosh(kh)}
{(1+0)^{\alpha+\beta}}\ee^{\pi\alpha\sinh(kh)}\\
&\quad +h\sum_{k=N+1}^{\infty}
\frac{K\ee^{\pi(1-\alpha)/12}\left\{\pi\sinh(kh)-\log(\log 2)\right\}\pi\cosh(kh)}
{(1+0)^{\alpha+\beta}}\ee^{-\pi\beta\sinh(kh)}.
\end{align*}
Using Proposition~\ref{prop:DE1-monotone},
for $Mh\geq\arsinh(2/(\pi\alpha))$ and $Nh\geq\arsinh(2/(\pi\beta))$,
we further bound the sums as
\begin{align*}
&h\sum_{k=-\infty}^{-M-1}
K\ee^{\pi(1-\alpha)/12}\left\{-\pi\sinh(kh)-\log(\log 2)\right\}\pi\cosh(kh)
\ee^{\pi\alpha\sinh(kh)}\\
&\quad +h\sum_{k=N+1}^{\infty}
K\ee^{\pi(1-\alpha)/12}\left\{\pi\sinh(kh)-\log(\log 2)\right\}\pi\cosh(kh)
\ee^{-\pi\beta\sinh(kh)}\\
&\leq K\ee^{\pi(1-\alpha)/12}\int_{-\infty}^{-Mh}
\left\{-\pi\sinh x-\log(\log 2)\right\}\pi\cosh x
\ee^{\pi\alpha\sinh x}\dd{x}\\
&\quad +K\ee^{\pi(1-\alpha)/12}\int_{Nh}^{\infty}
\left\{\pi\sinh x-\log(\log 2)\right\}\pi\cosh x
\ee^{-\pi\beta\sinh x}\dd{x}\\
&=\frac{K\ee^{\pi(1-\alpha)/12}}{\alpha}
\left(\pi\sinh(Mh)-\log(\log 2) + \frac{1}{\alpha}\right)
\ee^{-\pi\alpha\sinh(Mh)}\\
&\quad +\frac{K\ee^{\pi(1-\alpha)/12}}{\beta}
\left(\pi\sinh(Nh)-\log(\log 2) + \frac{1}{\beta}\right)
\ee^{-\pi\beta\sinh(Nh)}\\
&\leq\frac{K\ee^{\pi(1-\alpha)/12}}{\mu}
\left(\pi\sinh(Mh)-\log(\log 2) + \frac{1}{\mu}\right)
\ee^{-\pi\alpha\sinh(Mh)}\\
&\quad +\frac{K\ee^{\pi(1-\alpha)/12}}{\mu}
\left(\pi\sinh(Nh)-\log(\log 2) + \frac{1}{\mu}\right)
\ee^{-\pi\beta\sinh(Nh)},
\end{align*}
where $\mu=\min\{\alpha,\beta\}$ is used
at the last inequality.
Finally, using~\eqref{eq:def-MN-DE},
we have
$\ee^{-\pi\alpha\sinh(M h)}\leq \ee^{-\pi\mu q(2 d n/\mu)}$
(because $\alpha\sinh(Mh)\geq \mu q(2 d n/\mu)$)
and $\ee^{-\pi\beta\sinh(N h)}\leq \ee^{-\pi\mu q(2 d n/\mu)}$
(because $\beta\sinh(Nh)\geq \mu q(2 d n/\mu)$),
from which we obtain the conclusion.
\end{proof}

We are now in a position to prove Theorem~\ref{thm:new-DE3}.
Note that if $h\leq \pi d$,
then with $h$ selected by~\eqref{eq:improve-h-DE},
the inequalities~\eqref{eq:Mh-bounded-by-arsinh}
and~\eqref{eq:Nh-bounded-by-arsinh} hold.
Furthermore, if $n\geq \mu\sinh(1)/(2d)$,
i.e., if $\arsinh(2dn/\mu)\geq 1$,
then with $h$ selected by~\eqref{eq:improve-h-DE},
the inequality~\eqref{eq:M-leq-n} holds.
Similarly, it holds that
\begin{align*}
N &= \left\lceil
\frac{1}{h}\arsinh\left(\frac{2 d n/\beta}{\arsinh(2 d n/\mu)}\right)
\right\rceil\\
&=\left\lceil
\frac{n}{\arsinh(2 d n/\mu)}\arsinh\left(\frac{2 d n/\beta}{\arsinh(2 d n/\mu)}\right)
\right\rceil\\
&\leq\left\lceil
\frac{n}{\arsinh(2 d n/\mu)}\arsinh\left(\frac{2 d n/\beta}{1}\right)
\right\rceil\\
&\leq\left\lceil
\frac{n}{\arsinh(2 d n/\mu)}\arsinh\left(2 d n/\mu\right)
\right\rceil\\
&=n.
\end{align*}
From the inequalities, the truncation error
in Lemma~\ref{lem:DE3-truncation} is further bounded as
\begin{align*}
&\frac{K\ee^{\pi(1-\alpha)/12}}{\mu}
\left(\pi\sinh(Mh) + \pi\sinh(Nh) - 2\log(\log 2) + \frac{2}{\mu}\right)
\ee^{-\pi\mu q(2 d n/\mu)}\\
&\leq\frac{K\ee^{\pi(1-\alpha)/12}}{\mu}
\left(\pi\sinh(nh) + \pi\sinh(nh) - 2\log(\log 2) + \frac{2}{\mu}\right)
\ee^{-\pi\mu q(2 d n/\mu)}\\
&=\frac{2K\ee^{\pi(1-\alpha)/12}}{\mu}
\left(\pi\sinh(nh) - \log(\log 2) + \frac{1}{\mu}\right)
\ee^{-\pi\mu q(2 d n/\mu)}.
\end{align*}
Therefore,
from Lemmas~\ref{lem:DE3-discretization}
and~\ref{lem:DE3-truncation},
substituting~\eqref{eq:improve-h-DE} into $h$,
we have
\begin{align*}
\left|
\int_{-\infty}^{\infty}F(x)\dd{x}
- h \sum_{k=-M}^N F(kh)
\right| \leq C(n) n\ee^{-\pi\mu q(2 d n/\mu)},
\end{align*}
where
\begin{align*}
C(n) &=\frac{2K}{\mu^2}
\Biggl[
\frac{2 L_d^{1-\alpha}\left\{
(1+c_d)(1+d)(1+\pi\mu\cos d) - \mu\log(\log 2)\log(2+c_d)\cos d
\right\}}{n(1 - \ee^{-\pi\mu q(2dn/\mu)})\log(2+c_d)\cos^{\alpha+\beta}((\pi/2)\sin d)\cos^2 d}
\Biggr.\\
& \quad\quad\quad\quad
\Biggl.
+ \ee^{\pi(1-\alpha)/12}\left\{2\pi d + 1 - \frac{\mu\log(\log 2)}{n}\right\}
\Biggr].
\end{align*}
Furthermore, $C(n)\leq C(1)$ holds, which completes the proof of
Theorem~\ref{thm:new-DE3}.

\backmatter

%
%
%

\bmhead{Acknowledgments}
This work was partially supported by the
JSPS Grant-in-Aid for Scientific Research (C)
Number JP23K03218.

\bmhead{Data availability}
The data that support the findings of this study are available
from the corresponding author upon reasonable request.

\section*{Statements and Declarations}

\begin{itemize}
 \item \textbf{Competing Interests:}
 The authors have no competing interests to declare
 that are relevant to the content of this article.
\end{itemize}


\bibliography{sn-bibliography}


\end{document}